\documentclass{amsart}

\pdfoutput=1

\usepackage{amsmath}
\usepackage{amssymb}
\usepackage{bm}
\usepackage{bbm}
\usepackage{enumerate}
\usepackage{booktabs} 
\usepackage{appendix}
\usepackage{mathrsfs}
\usepackage{hyperref}
\usepackage{graphicx, color}
\hypersetup{
  pdftitle={Asymptotic error in the eigenfunction expansion for the Green's
    function of a Sturm--Liouville problem},
  pdfauthor={Karen Habermann},
}

\usepackage[T1]{fontenc}
\usepackage[utf8]{inputenc}

\usepackage{xcolor}

\numberwithin{equation}{section}

\newcommand{\R}{\mathbb{R}}
\newcommand{\N}{\mathbb{N}}

\newcommand{\E}{\mathbb{E}}

\newcommand{\e}{\operatorname{e}}
\newcommand{\dd}{\,{\mathrm d}}
\newcommand{\db}{{\mathrm d}}

\newcommand{\lo}{\mathcal{L}}

\newcommand{\eps}{\varepsilon}
\newcommand{\pt}{\partial}

\newcommand{\sqp}{\sqrt{\pi}\,}

\newtheorem{lemma}{Lemma}[section]

\newtheorem{propn}[lemma]{Proposition}
\newtheorem{thm}[lemma]{Theorem}
\newtheorem{cor}[lemma]{Corollary}

\newtheorem{remark}[lemma]{Remark}
\newtheorem{eg0}[lemma]{Example}

\makeatletter
\@namedef{subjclassname@2020}{%
  \textup{2020} Mathematics Subject Classification}
\makeatother
\author[K. Habermann]{Karen Habermann}
\address{Karen Habermann, Department of Statistics, University of
  Warwick, Coventry, CV4 7AL, United Kingdom.}
\email{karen.habermann@warwick.ac.uk}

\title[Asymptotic error in the eigenfunction expansion for the Green's
  function]{Asymptotic error in the eigenfunction expansion for the Green's
    function of a Sturm--Liouville problem}
\begin{document}
\begin{abstract}
  We study the asymptotic error arising when approximating the
  Green's function of a Sturm--Liouville problem through a truncation of its
  eigenfunction expansion, both for the Green's function of a regular
  Sturm--Liouville problem and for the Green's function associated
  with the Hermite polynomials, the associated
  Laguerre polynomials, and the Jacobi polynomials, respectively.
  We prove that the asymptotic error obtained on the diagonal can be
  expressed in terms of the coefficients of the related second-order
  Sturm--Liouville differential equation, and that the suitable scaling
  exponent which yields a non-degenerate limit on the diagonal depends
  on the asymptotic behaviour of the corresponding eigenvalues. We
  further consider the asymptotic error away from the diagonal and
  analyse which scaling exponents ensure that it
  remains at zero. For the Hermite polynomials, the associated
  Laguerre polynomials, and the Jacobi polynomials, a
  Christoffel--Darboux type formula, which we establish for all
  classical orthogonal polynomial systems, allows us to obtain a
  better control away from the diagonal than what a sole application of
  known asymptotic formulae gives. As a consequence of our study
  for regular Sturm--Liouville problems, we identify the
  fluctuations for the Karhunen--Lo{\`e}ve expansion of
  Brownian motion.
\end{abstract}

\maketitle
\thispagestyle{empty}
%%%%%%%%%%%%%%%%%%%%%%%%%%%%%%%%%%%%%%%%%%%%%%%%%%%%%%%%%%%%%%%%%%%%%%%%%%%%%%%
\section{Introduction}

In a regular Sturm--Liouville problem, we are concerned
with a real second-order linear differential operator
\begin{displaymath}
  \lo =
  -\frac{1}{w(x)}\left(
    \frac{\db}{\db x}\left(p(x)\frac{\db}{\db x}\right) -q(x)
  \right)
\end{displaymath}
for a positive continuously
differentiable function $p\colon [a,b]\to\R$, a continuous
function $q\colon [a,b]\to\R$ as well as a positive continuous function
$w\colon [a,b]\to\R$ defined on a finite interval $I=[a,b]$, and we are
interested in solving the eigenvalue equation, for $x\in[a,b]$,
\begin{displaymath}
  \lo \phi(x)=\lambda \phi(x)
\end{displaymath}
subject to the separated homogeneous boundary conditions
\begin{align*}
  \alpha_1 \phi(a)+\alpha_2 \phi'(a)&=0\;,
  \qquad\alpha_1^2+\alpha_2^2>0\;,\\
  \beta_1 \phi(b)+\beta_2 \phi'(b)&=0\;,
  \qquad\beta_1^2+\beta_2^2>0\;,
\end{align*}
for constants $\alpha_1,\alpha_2,\beta_1,\beta_2\in\R$. Non-zero
solutions $\phi$ to the Sturm--Liouville problem are called
eigenfunctions with $\lambda$ being the corresponding eigenvalue.

Since the early works by Sturm and Liouville~\cite{SL}, the theory of
Sturm--Liouville problems has seen a vast development with further
contributions, among many others, due to Weyl~\cite{weyl},
Dixon~\cite{dixon}, and Titchmarsh~\cite{titch1,titch2}.
It is by now well-known that, for a given Sturm--Liouville problem,
there exists a family $\{\phi_n: n\in\N_0\}$ of real eigenfunctions that
forms a complete
orthogonal set under the weighted inner product in the Hilbert space
$L^2([a,b], w(x)\dd x)$ which can be ordered such that the
corresponding real eigenvalues $\{\lambda_n: n\in\N_0\}$ form a sequence
which strictly increases to infinity. The completeness and
orthogonality relation of the eigenfunctions can be expressed in the sense
of distributions in terms
of the Dirac delta function as, for $x,y\in[a,b]$,
\begin{displaymath}
  \sum_{n=0}^\infty
  \frac{\phi_n(x)\phi_n(y)}{\int_a^b\left(\phi_n(z)\right)^2w(z)\dd z}
  =\delta(x-y)\;.
\end{displaymath}
Moreover, provided all eigenvalues are non-zero, the Green's function
$G\colon[a,b]\times[a,b]\to\R$ of the
linear differential operator $\lo$ subject to the separated homogeneous
boundary conditions stated above has the eigenfunction expansion
given by, for $x,y\in[a,b]$,
\begin{displaymath}
  G(x,y)=\sum_{n=0}^\infty
  \frac{\phi_n(x)\phi_n(y)}
  {\lambda_n\int_a^b\left(\phi_n(z)\right)^2w(z)\dd z}\;.
\end{displaymath}
This expansion is also referred to as the bilinear expansion of the
Green's function. Formally, we see that indeed
\begin{displaymath}
  \lo G(x,y)=\sum_{n=0}^\infty
  \frac{\lo\phi_n(x)\phi_n(y)}
  {\lambda_n\int_a^b\left(\phi_n(z)\right)^2w(z)\dd z}
  =\sum_{n=0}^\infty
  \frac{\phi_n(x)\phi_n(y)}
  {\int_a^b\left(\phi_n(z)\right)^2w(z)\dd z}
  =\delta(x-y)\;.
\end{displaymath}
The present article is concerned with studying the pointwise limit
\begin{equation}\label{eq:green_fluct}
  \lim_{N\to\infty} N^\gamma\sum_{n=N+1}^\infty\frac{\phi_n(x)\phi_n(y)}
  {\lambda_n\int_a^b\left(\phi_n(z)\right)^2w(z)\dd z}\;,
\end{equation}
provided the eigenvalues are ordered in increasing order, and
for a suitable exponent $\gamma>0$ which may take a different value
on the diagonal $\{x=y\}$ than away from the diagonal.
The resulting limit quantifies the asymptotic error in approximating
the Green's function $G$ with its truncation $G_N$ defined by, for
$x,y\in[a,b]$,
\begin{displaymath}
  G_N(x,y)=\sum_{n=0}^N
  \frac{\phi_n(x)\phi_n(y)}
  {\lambda_n\int_a^b\left(\phi_n(z)\right)^2w(z)\dd z}
\end{displaymath}
as $N\to\infty$. We further observe that the limit~(\ref{eq:green_fluct}) of
interest can be considered both for regular Sturm--Liouville
problems which admit a zero eigenvalue and for singular Sturm--Liouville
problems which admit a complete set of orthogonal eigenfunctions.

In a
singular Sturm--Liouville problem, the interval $I$ on which the
problem is studied may
be infinite, in which case suitable singular boundary conditions
are imposed as discussed in Littlejohn and Krall~\cite{littlekrall},
the function $p$ may vanish on $I$ and the function $w$
may vanish or be ill-defined at the boundary points of $I$.
A rich class of functions arising as solutions to singular
Sturm--Liouville problems are the classical orthogonal polynomials,
which are related by linear transformations to the Hermite
polynomials, the associated Laguerre polynomials,
and the Jacobi polynomials. As
conjectured by Acz\'{e}l~\cite{aczel} and as established by
Hahn~\cite{hahn}, Feldmann~\cite{feldmann}, Mikol\'{a}s~\cite{mikolas}
and Lesky~\cite{lesky}, the classical orthogonal polynomial systems are the
only orthogonal polynomial systems which arise as solutions to
second-order Sturm--Liouville
differential equations.

In our analysis of the asymptotic error in the eigenfunction expansion
for the Green's function, we start by determining the
limit~(\ref{eq:green_fluct}) for all regular Sturm--Liouville
problems where
the coefficients of the linear differential operator $\lo$ satisfy
$p,w\in C^2([a,b])$. This additional requirement is
currently needed as our proof employs the Liouville
transformation to reduce the system to a simplified Sturm--Liouville
problem for which explicit asymptotic formulae for the eigenfunctions
and the eigenvalues are known. It would be of interest to study in
future work if this assumption can be relaxed.
\begin{thm}\label{thm:regularSL}
  Fix $a,b\in\R$. For a function $q\in C([a,b])$ and positive functions $p,w\in
  C^2([a,b])$, let $\{\phi_n:n\in\N_0\}$ be a complete orthogonal set 
  of eigenfunctions for the regular Sturm--Liouville problem
  \begin{displaymath}
    \frac{\db}{\db x}\left(p(x)\frac{\db\phi(x)}{\db x}\right) -q(x)\phi(x)
    =-\lambda w(x)\phi(x)
  \end{displaymath}
  on the finite interval $[a,b]$ subject to separated homogeneous
  boundary conditions.
  Assume that the eigenfunctions are ordered such that the
  corresponding eigenvalues $\{\lambda_n:n\in\N_0\}$ form a strictly
  increasing sequence. We have, for $x\in (a,b)$,
  \begin{displaymath}
    \lim_{N\to\infty} N\sum_{n=N+1}^\infty \frac{\left(\phi_n(x)\right)^2}
    {\lambda_n\int_a^b\left(\phi_n(z)\right)^2w(z)\dd z}
    =\frac{1}{\pi^2\sqrt{p(x)w(x)}}\int_a^b\sqrt{\frac{w(z)}{p(z)}}\dd z
  \end{displaymath}
  and, for $x,y\in [a,b]$ with $x\not= y$,
  \begin{displaymath}
    \lim_{N\to\infty} N\sum_{n=N+1}^\infty \frac{\phi_n(x)\phi_n(y)}
    {\lambda_n\int_a^b\left(\phi_n(z)\right)^2w(z)\dd z}
    =0\;.
  \end{displaymath}
  Moreover, depending on the separated homogeneous boundary
  conditions, we have
  \begin{displaymath}
    \lim_{N\to\infty} N\sum_{n=N+1}^\infty \frac{\left(\phi_n(a)\right)^2}
    {\lambda_n\int_a^b\left(\phi_n(z)\right)^2w(z)\dd z}=
    \begin{cases}
      \displaystyle{\frac{2}{\pi^2\sqrt{p(a)w(a)}}
      \int_a^b\sqrt{\frac{w(z)}{p(z)}}\dd z}
      & \text{if } \alpha_2\not = 0\\[0.7em]
      0 & \text{if }\alpha_2=0
    \end{cases}
  \end{displaymath}
  as well as
  \begin{displaymath}
    \lim_{N\to\infty} N\sum_{n=N+1}^\infty \frac{\left(\phi_n(b)\right)^2}
    {\lambda_n\int_a^b\left(\phi_n(z)\right)^2w(z)\dd z}=
    \begin{cases}
      \displaystyle{\frac{2}{\pi^2\sqrt{p(b)w(b)}}
      \int_a^b\sqrt{\frac{w(z)}{p(z)}}\dd z}
      & \text{if } \beta_2\not = 0\\[0.7em]
      0 & \text{if }\beta_2=0
    \end{cases}\;.
  \end{displaymath}
\end{thm}
While~\cite[Theorem~1.2]{foster_habermann}, which states that, for all
$s,t\in[0,1]$,
\begin{equation}\label{eq:baseDD}
  \lim_{N\to\infty}N
  \sum_{n=N+1}^\infty\frac{2\sin((n+1)\pi s)\sin((n+1)\pi t)}{(n+1)^2\pi^2}=
  \begin{cases}
    \dfrac{1}{\pi^2} & \text{if } s=t\text{ and } t\in(0,1)\\[0.7em]
    0 & \text{otherwise}
  \end{cases}\;,
\end{equation}
could be considered as a special case of Theorem~\ref{thm:regularSL},
our proof actually depends on this result. Together with
Proposition~\ref{propn:baseDN}, it forms one of the base cases which
the problem reduces to after applying
the Liouville transformation as well as the known asymptotic formulae, and
which needs to be proven separately.

As a consequence of Proposition~\ref{propn:baseDN}, we further obtain
the fluctuations for the Karhunen--Lo{\`e}ve expansion of Brownian
motion, just as~\cite[Theorem~1.2]{foster_habermann} allowed us to
determine the
fluctuations for the Karhunen--Lo{\`e}ve expansion of a Brownian
bridge.
\begin{thm}\label{thm:KLBM}
  Let $(B_t)_{t\in[0,1]}$ be a Brownian motion in $\R$ and, for
  $N\in\N_0$, let the fluctuation processes $(F_t^{N})_{t\in[0,1]}$ for the
  Karhunen--Lo{\`e}ve expansion be defined by, for $t\in[0,1]$,
  \begin{displaymath}
    F_t^{N}
    =\sqrt{N}\left(B_t-
      \sum_{n=0}^N\frac{2\sin\left(\left(n+\frac{1}{2}\right)\pi t\right)}
      {\left(n+\frac{1}{2}\right)\pi}
      \int_0^1\cos\left(\left(n+\frac{1}{2}\right)\pi r\right)\dd B_r\right)\;.
  \end{displaymath}
  The processes $(F_t^{N})_{t\in[0,1]}$
  converge in finite dimensional distributions as $N\to\infty$ to the
  collection $(F_t)_{t\in[0,1]}$ of independent Gaussian random
  variables with mean zero and variance
  \begin{displaymath}
    \E\left[\left(F_t\right)^2\right]=
    \begin{cases}
      \dfrac{1}{\pi^2} & \text{if }t\in(0,1)\\[0.7em]
      \dfrac{2}{\pi^2} & \text{if }t=1\\[0.7em]
      0 & \text{if } t=0
    \end{cases}\;,
  \end{displaymath}
  where $\E$ denotes expectation with respect to Wiener measure.
\end{thm}

After proving Theorem~\ref{thm:regularSL} and illustrating the result
by one example, we turn our attention to singular Sturm--Liouville
problems with a focus on classical orthogonal polynomial systems,
which arise as solutions to
\begin{displaymath}
  \frac{\db}{\db x}\left(P(x)\frac{\db Y(x)}{\db x}\right)
  =-\lambda W(x) Y(x)
\end{displaymath}
subject to suitable singular boundary conditions.
The Hermite polynomial $H_n\colon\R\to\R$ of degree $n\in\N_0$
satisfies the Sturm--Liouville differential equation
\begin{displaymath}
  \left(\e^{-x^2}H_n'(x)\right)'=-2n\e^{-x^2}H_n(x)\;.
\end{displaymath}
The asymptotic error in approximating the associated Green's function
by truncating its bilinear expansion is given by the following result.
\begin{thm}\label{thm:hermite}
  We have, for $x\in\R$,
  \begin{displaymath}
    \lim_{N\to\infty}\sqrt{N}\sum_{n=N+1}^\infty
    \frac{\left(H_n(x)\right)^2}{2n\,2^nn!\sqrt{\pi}}
    =\frac{\exp\left(x^2\right)}{\sqrt{2}\pi}
  \end{displaymath}
  as well as, for $x,y\in\R$ with $x\not= y$ and for all $\gamma<1$,
  \begin{displaymath}
    \lim_{N\to\infty}N^\gamma\sum_{n=N+1}^\infty
    \frac{H_n(x)H_n(y)}{2n\,2^nn!\sqrt{\pi}}=0\;.
  \end{displaymath}
\end{thm}
For $\alpha\in\R$ fixed with $\alpha>-1$, the
associated Laguerre polynomial $L_n^{(\alpha)}\colon [0,\infty)\to\R$
of degree $n\in\N_0$ solves the Sturm--Liouville differential equation
\begin{displaymath}
  \frac{\db}{\db x}
  \left(x^{\alpha+1}\e^{-x}\frac{\db L_n^{(\alpha)}(x)}{\db x}\right)
  =-n x^\alpha \e^{-x}L_n^{(\alpha)}(x)\;.
\end{displaymath}
The special case $\alpha=0$ gives rise to the Laguerre polynomials.
As for the Hermite polynomials in Theorem~\ref{thm:hermite}, and unlike in
Theorem~\ref{thm:regularSL}, we need to rescale the error in the
approximations for the Green's function by $\sqrt{N}$ to obtain a
non-trivial limit on the diagonal. This is essentially a consequence
of both Hermite polynomials and associated Laguerre polynomials having
eigenvalues satisfying $\lambda_n=O(n)$ as $n\to\infty$ as opposed to
$\lambda_n=O(n^2)$ occurring in all other cases considered.
\begin{thm}\label{thm:laguerre}
  For $\alpha\in\R$ fixed with $\alpha>-1$,
  we have, for $x\in(0,\infty)$,
  \begin{displaymath}
    \lim_{N\to\infty}\sqrt{N}\sum_{n=N+1}^\infty
    \frac{\Gamma(n)\left(L_n^{(\alpha)}(x)\right)^2}
    {\Gamma(n+\alpha+1)}
    =\frac{x^{-\alpha-\frac{1}{2}} \e^x}{\pi}
  \end{displaymath}
  and, for $x,y\in(0,\infty)$ with $x\not= y$ and for all $\gamma<\frac{1}{2}$,
  \begin{displaymath}
    \lim_{N\to\infty} N^\gamma\sum_{n=N+1}^\infty
    \frac{\Gamma(n)L_n^{(\alpha)}(x)L_n^{(\alpha)}(y)}
    {\Gamma(n+\alpha+1)}
    =0\;.
  \end{displaymath}
\end{thm}
For the associated Laguerre polynomials, we cannot prove the
off-diagonal convergence
up to the scaling exponent appearing on the diagonal. The result
stated is the best we can achieve with our method, discussed later, and
using known asymptotic formulae for the
polynomials as $n\to\infty$.

Further notice that both for Hermite polynomials and for associated
Laguerre polynomials the limit function formed on the diagonal is
still proportional to $1/\sqrt{PW}$, as in
Theorem~\ref{thm:regularSL}, but with different constants of
proportionality.

For $\alpha,\beta\in\R$ fixed with $\alpha,\beta>-1$, the Jacobi
polynomial $P_n^{(\alpha,\beta)}\colon [-1,1]\to\R$ of degree
$n\in\N_0$ satisfies the Sturm--Liouville differential equation
\begin{displaymath}
  \frac{\db}{\db x}\left((1-x)^{\alpha+1}(1+x)^{\beta+1}
    \frac{\db P_n^{(\alpha,\beta)}(x)}{\db x}\right)
  =-n(n+\alpha+\beta+1) (1-x)^\alpha(1+x)^\beta P_n^{(\alpha,\beta)}(x)\;,
\end{displaymath}
and the asymptotic error in the eigenfunction expansion for the
associated Green's function is given in the theorem below. We see
that the limit function formed on the diagonal is exactly the same one
which appears in Theorem~\ref{thm:regularSL} for regular Sturm--Liouville
problems as
\begin{displaymath}
  \int_{-1}^1\sqrt{\frac{W(z)}{P(z)}}\dd z
  =\int_{-1}^1\frac{1}{\sqrt{1-z^2}}\dd z=\pi\;.
\end{displaymath}
The analysis we employ for the Jacobi polynomials in fact allows us to
determine the asymptotic error even for those families of Jacobi
polynomials which do not arise as classical orthogonal polynomial
systems
but, with $\alpha,\beta\in\R$, are defined by, for $n\in\N_0$ and
$x\in[-1,1]$,
\begin{displaymath}
  P_n^{(\alpha,\beta)}(x)
  =\frac{1}{n!}\sum_{k=0}^n\binom{n}{k}
  \left(n+\alpha+\beta+1\right)_k\left(\alpha+k+1\right)_{n-k}
  \left(\frac{x-1}{2}\right)^k\;.
\end{displaymath}
\begin{thm}\label{thm:jacobi}
  For $\alpha,\beta\in\R$ fixed,
  we have, for $x\in(-1,1)$,
  \begin{displaymath}
    \lim_{N\to\infty}N\sum_{n=N+1}^\infty
    \frac{n!(2n+\alpha+\beta+1)\Gamma(n+\alpha+\beta+1)}
    {2^{\alpha+\beta+1}\Gamma(n+\alpha+1)\Gamma(n+\beta+1)}
      \frac{\left(P_n^{(\alpha,\beta)}(x)\right)^2}
      {n(n+\alpha+\beta+1)}
    =\dfrac{(1-x)^{-\alpha-\frac{1}{2}}(1+x)^{-\beta-\frac{1}{2}}}{\pi}
  \end{displaymath}
  and, for $x,y\in(-1,1)$ with $x\not= y$ and for all $\gamma<2$,
  \begin{displaymath}
    \lim_{N\to\infty}N^\gamma\sum_{n=N+1}^\infty
    \frac{n!(2n+\alpha+\beta+1)\Gamma(n+\alpha+\beta+1)}
    {2^{\alpha+\beta+1}\Gamma(n+\alpha+1)\Gamma(n+\beta+1)}
      \frac{P_n^{(\alpha,\beta)}(x)P_n^{(\alpha,\beta)}(y)}
      {n(n+\alpha+\beta+1)}=0\;.
  \end{displaymath}
\end{thm}
We show later that~\cite[Theorem~1.5]{semicircle} which
quantifies 
an integrated version of the completeness and orthogonality property
for Legendre polynomials can be deduced from
Theorem~\ref{thm:jacobi}.
Moreover, we have the following
convergence results for the Legendre polynomials $\{P_n:n\in\N_0\}$,
the Chebyshev polynomials of the first kind $\{T_n:n\in\N_0\}$ and the
Chebyshev polynomials of the second kind $\{U_n:n\in\N_0\}$, which all
are scalar multiplies of Jacobi polynomials.
\begin{cor}\label{cor:jacobi}
  We have, for $x,y\in(-1,1)$,
  \begin{displaymath}
    \lim_{N\to\infty}N\sum_{n=N+1}^\infty
    \frac{(2n+1)P_n(x)P_n(y)}{2n(n+1)}=
    \begin{cases}
      \dfrac{1}{\pi\sqrt{1-x^2}} & \text{if }x=y\\[0.7em]
      0 &\text{if }x\not=y
    \end{cases}
  \end{displaymath}
  as well as
  \begin{displaymath}
    \lim_{N\to\infty}N\sum_{n=N+1}^\infty
    \frac{T_n(x)T_n(y)}{n^2}=
    \begin{cases}
      \frac{1}{2} & \text{if }x=y\\
      0 &\text{if }x\not=y
    \end{cases}
  \end{displaymath}
  and
  \begin{displaymath}
    \lim_{N\to\infty}N\sum_{n=N+1}^\infty
    \frac{U_n(x)U_n(y)}{n(n+2)}=
    \begin{cases}
      \dfrac{1}{2(1-x^2)} & \text{if }x=y\\[0.7em]
      0 &\text{if }x\not=y
    \end{cases}\;.
  \end{displaymath}
\end{cor}
When proving the convergence results for the classical orthogonal
polynomials, we use, in principle, the same method which was employed to
establish~\cite[Theorem~1.5]{semicircle} and we split the analysis into
an on-diagonal and an off-diagonal part.
The pointwise convergence on the diagonal
is deduced from a convergence of moments, and two local uniform bounds
which allow us to apply the Arzel\`{a}–Ascoli theorem. 
For the
convergence away from the diagonal, we find a Christoffel--Darboux
type formula, see Proposition~\ref{propn:CDT}, which together with
known asymptotic formulae for the polynomials yields the desired
result. By using the Christoffel--Darboux
type formula, we obtain a
better control away from the diagonal than a sole application of
the asymptotic formulae would give.

The main advantage compared to~\cite{semicircle} is that we have a
streamlined way to deal with the moment convergence.
Proposition~\ref{propn:moments} and Proposition~\ref{propn:conv_mom}
imply that for a family of classical orthogonal polynomials which is
subject to suitable boundary conditions the limit moments on the
diagonal satisfy the same recurrence relation as the moments of the desired
limit function. The one subtlety which appears to remain, and which
forces us to employ a workaround for the Jacobi polynomials, is
checking
that the limit moments satisfy the right initial condition. In doing
so, we establish a number of interesting identities for Hermite
polynomials and associated Laguerre polynomials.

We remark that the used method of moments, see
Billingsley~\cite[Theorem~30.2]{billingsley}, is applicable in our
analysis since each limit function encountered is determined by
its moments as a consequence of~\cite[Theorem~30.1]{billingsley}.

\subsection*{\bf The paper is organised as follows.}
After starting by reviewing the Liouville transformation for
Sturm--Liouville problems,
we give an account, in Section~\ref{sec:asym}, of the asymptotic
formulae for eigenvalues 
and eigenfunctions of systems in Liouville normal form.
We then prove Theorem~\ref{thm:regularSL} and
Proposition~\ref{propn:baseDN}, which implies Theorem~\ref{thm:KLBM},
in Section~\ref{sec:main},
and we illustrate Theorem~\ref{thm:regularSL} by one example given in
Section~\ref{sec:1example}. Our considerations for classical
orthogonal polynomials follow in Section~\ref{sec:COP}. We first
prove Proposition~\ref{propn:moments} and
Proposition~\ref{propn:conv_mom}, which allow us to streamline the
moment analysis on the diagonal, and we further derive the
Christoffel--Darboux type formula stated in
Proposition~\ref{propn:CDT}, applicable to any family of classical
orthogonal polynomials, which is used for the off-diagonal
analysis. We apply these results to prove
Theorem~\ref{thm:hermite} in Section~\ref{sec:hermite},
Theorem~\ref{thm:laguerre} in Section~\ref{sec:laguerre},
and finally, Theorem~\ref{thm:jacobi} in Section~\ref{sec:jacobi}. Our
proof of Theorem~\ref{thm:jacobi} relies on
Proposition~\ref{propn:generalbc}, which demonstrates that
Theorem~\ref{thm:regularSL} extends to cases beyond separated
homogeneous boundary conditions.
Throughout, we use the convention that $\N$ denotes the positive
integers and $\N_0$ the non-negative integers.

\section{Green's function of a regular Sturm--Liouville
  problem}
\label{sec:regular}

We prove Theorem~\ref{thm:regularSL} by applying the Liouville
transformation to reduce the analysis to systems in Liouville normal
form, for which the asymptotic error in the eigenfunction expansion for
the Green's function is
deduced from the asymptotics for the eigenfunctions as well as the
eigenvalues and from four base cases which we establish directly.

The reason for assuming $p,w\in C^2([a,b])$ in
Theorem~\ref{thm:regularSL} is that
a given Sturm--Liouville differential equation
\begin{displaymath}
  \frac{\db}{\db x}\left(p(x)\frac{\db\phi(x)}{\db x}\right) -q(x)\phi(x)
  =-\lambda w(x)\phi(x)
\end{displaymath}
is transformed by the Liouville transformation
\begin{displaymath}
  t=\int_a^x\sqrt{\frac{w(z)}{p(z)}}\dd z
  \quad\text{and}\quad
  u(t)=\sqrt[4]{p(x)w(x)}\phi(x)
\end{displaymath}
and with
\begin{displaymath}
  \tilde{q}(t)=\frac{q(x)}{w(x)}+
  \frac{1}{\sqrt[4]{p(x)w(x)}}\frac{\db^2}{\db t^2}
  \left(\sqrt[4]{p(x)w(x)}\right)
\end{displaymath}
into the Liouville normal form
\begin{displaymath}
  \frac{\db^2u(t)}{\db t^2}-\tilde{q}(t)u(t)
  =-\lambda u(t)
\end{displaymath}
on the finite interval $[0,c]$ where
\begin{displaymath}
  c=\int_a^b\sqrt{\frac{w(z)}{p(z)}}\dd z\;,
\end{displaymath}
see Birkhoff and Rota~\cite[Theorem~10.6]{birkhoff}.
In particular, the additional assumption $p,w\in C^2([a,b])$ together
with the standard assumptions of $p$ and $w$ being positive
and $q\in C([a,b])$ ensures that $\tilde{q}\in C([0,c])$.
According to~\cite[Corollary~10.1 and Corollary~10.2]{birkhoff}, the
Liouville transformation transforms the original regular
Sturm--Liouville problem with
separated homogeneous boundary conditions into another regular
Sturm--Liouville problem with
separated homogeneous boundary conditions which has the same
eigenvalues as the original system and where eigenfunctions are mapped to
eigenfunctions, except that
the new eigenfunctions are now orthogonal with respect to a constant
weight function.

Hence, the Liouville transformation allows us to prove
Theorem~\ref{thm:regularSL} by analysing Sturm--Liouville problems in
Liouville normal form subject to separated homogeneous boundary
conditions.
For these systems asymptotic formulae for the eigenfunctions and
eigenvalues are known as presented in the subsequent section.

\subsection{Asymptotics for eigenfunctions and eigenvalues}
\label{sec:asym}

In our discussion of the asymptotic formulae for
eigenfunctions and eigenvalues of
systems in Liouville normal form, we follow the analysis in
Birkhoff and Rota~\cite[Section 10.10 to Section 10.13]{birkhoff}
with explicitly completing
the extension to all possible separated homogeneous boundary
conditions. For an analysis of all cases needing consideration,
one may also consult Levitan
and Sargsjan~\cite[Section~1.2]{levitan}.

Fix $c\in\R$ with $c>0$. We consider a regular Sturm--Liouville
problem in Liouville normal
form on the interval $[0,c]$, for $q\in C([0,c])$,
\begin{displaymath}
  \frac{\db^2u(t)}{\db t^2}-q(t)u(t)
  =-\lambda u(t)
\end{displaymath}
subject to the separated homogeneous boundary conditions
\begin{align*}
  \alpha_1 u(0)+\alpha_2 u'(0)&=0\;,
  \qquad\alpha_1^2+\alpha_2^2>0\;,\\
  \beta_1 u(c)+\beta_2 u'(c)&=0\;,
  \qquad\beta_1^2+\beta_2^2>0\;,
\end{align*}
for constants $\alpha_1,\alpha_2,\beta_1,\beta_2\in\R$. By
Sturm--Liouville theory, there exists a family $\{u_n: n\in\N_0\}$
of real eigenfunctions that forms a complete orthonormal set and such
that the corresponding real eigenvalues $\{\lambda_n: n\in\N_0\}$
form a sequence strictly increasing to infinity. In particular, as we
are interested in asymptotic formulae for $u_n$ and $\lambda_n$ as
$n\to\infty$, we may and do assume throughout our analysis that the function
$Q\in C([0,c])$ defined by
\begin{displaymath}
  Q(t)=\lambda-q(t)
\end{displaymath}
is strictly positive. For a given solution $u$ to the Sturm--Liouville
problem with eigenvalue $\lambda$, the modified amplitude $R$ and the
modified phase $\theta$ are defined by
\begin{equation}\label{eq:prufersub}
  u=\frac{R}{\sqrt[4]{Q}}\sin(\theta)
  \quad\text{and}\quad
  u'=R\sqrt[4]{Q}\cos(\theta)\;,
\end{equation}
which is called the modified Pr\"ufer substitution. We see that the
modified amplitude, chosen to be non-negative, and the modified phase,
subject to shifts by $2\pi$,
are determined by
the equations
\begin{displaymath}
  \left(R(t)\right)^2=
  \sqrt{Q(t)}\left(u(t)\right)^2+
  \frac{1}{\sqrt{Q(t)}}\left(u'(t)\right)^2
\end{displaymath}
as well as
\begin{align*}
  \cot\left(\theta(t)\right)=
  \frac{1}{\sqrt{Q(t)}}\frac{u'(t)}{u(t)}
  \quad
  &\text{provided }u(t)\not=0
    \intertext{and}
  \tan\left(\theta(t)\right)=
  \sqrt{Q(t)}\frac{u(t)}{u'(t)}
  \quad
  &\text{provided }u'(t)\not=0\;.
\end{align*}
Differentiating these equations and using $u''=-Qu$, we obtain the
modified Pr\"ufer system
\begin{align*}
  \theta'(t)
  &=\sqrt{\lambda-q(t)}
  -\frac{q'(t)}{4(\lambda-q(t))}
    \sin\left(2\theta(t)\right)\;,\\
  \frac{R'(t)}{R(t)}
  &=
  \frac{q'(t)}{4(\lambda-q(t))}
  \cos\left(2\theta(t)\right)\;.
\end{align*}
By comparing solutions to the modified Pr\"ufer system with solutions
to
\begin{displaymath}
  \tilde{\theta}'(t)=\sqrt{\lambda}
  \quad\text{and}\quad
  \left(\log \tilde{R}(t)\right)'=0\;,
\end{displaymath}
and by exploiting the property that, for given initial values, solutions
to an ordinary differential equation depend continuously on the
ordinary differential equation, we deduce that, as $\lambda\to\infty$,
\begin{equation}\label{eq:thetaandR}
  \theta(t)=\theta(0)+\sqrt{\lambda}t
  +O\left(\frac{1}{\sqrt{\lambda}}\right)
  \quad\text{and}\quad
  R(t)=R(0)+O\left(\frac{1}{\lambda}\right)\;,
\end{equation}
see proof of~\cite[Theorem~10.7]{birkhoff} for details. For the
remaining analysis it is necessary to distinguish according to the
type of separated homogeneous boundary conditions imposed, where we
essentially
differentiate if the Dirichlet part or the Neumann part of the
boundary conditions dominates for large $\lambda$.
This gives rise to four different cases.
\begin{propn}\label{propn:evals}
  The real eigenvalues $\{\lambda_n:n\in\N_0\}$
  ordered in increasing order
  of a regular Sturm--Liouville problem in Liouville normal form
  satisfy, as $n\to\infty$,
  \begin{equation*}
    c\sqrt{\lambda_n}+O\left(\frac{1}{n}\right)=
    \begin{cases}
      n\pi & \text{if }\alpha_2,\beta_2\not= 0\\
      \left(n+\frac{1}{2}\right)\pi
      &\text{if }\alpha_2\not= 0\text{ and }\beta_2=0\\
      \left(n+\frac{1}{2}\right)\pi
      &\text{if }\alpha_2=0\text{ and }\beta_2\not=0\\
      \left(n+1\right)\pi & \text{if }\alpha_2=\beta_2=0
    \end{cases}\;.
  \end{equation*}
\end{propn}
\begin{proof}
  Let us first consider the case $\alpha_2,\beta_2\not= 0$. For
  an eigenfunction corresponding to some large eigenvalue
  $\lambda_n$, the associated modified phase $\theta_n$ then satisfies
  \begin{displaymath}
    \cot\left(\theta_n(0)\right)
    =-\frac{\alpha_1}{\alpha_2}\frac{1}{\sqrt{\lambda_n-q(0)}}
    \quad\text{and}\quad
    \cot\left(\theta_n(c)\right)
    =-\frac{\beta_1}{\beta_2}\frac{1}{\sqrt{\lambda_n-q(c)}}\;.
  \end{displaymath}
  It follows that we can choose $\theta_n$
  such that, as $n\to\infty$,
  \begin{displaymath}
    \theta_n(0)=
    \frac{\pi}{2}+\frac{\alpha_1}{\alpha_2\sqrt{\lambda_n}}+
    O\left(\frac{1}{\sqrt{\lambda_n^3}}\right)\;.
  \end{displaymath}
  Since the eigenfunction corresponding to the eigenvalue $\lambda_n$
  has exactly $n$ zeros in the interval $(0,c)$,
  see~\cite[Theorem~10.5]{birkhoff}, we further obtain that
  \begin{displaymath}
    \theta_n(c)=
    \frac{\pi}{2}+n\pi+O\left(\frac{1}{\sqrt{\lambda_n}}\right)\;.
  \end{displaymath}
  This implies that
  \begin{displaymath}
    \theta_n(c)-\theta_n(0)=
    n\pi+O\left(\frac{1}{\sqrt{\lambda_n}}\right)
  \end{displaymath}
  and a comparison with~(\ref{eq:thetaandR}) shows that we indeed
  have, as $n\to\infty$,
  \begin{displaymath}
    c\sqrt{\lambda_n}=n\pi+O\left(\frac{1}{n}\right)\;.
  \end{displaymath}
  We argue similarly in the three remaining cases, except we need
  to take care of phase shifts by $\frac{\pi}{2}$. If
  $\alpha_2\not =0$ and $\beta_2=0$ then, as before, we can choose the
  modified phase such that, as $n\to\infty$,
  \begin{displaymath}
    \theta_n(0)=
    \frac{\pi}{2}+O\left(\frac{1}{\sqrt{\lambda_n}}\right)\;,
  \end{displaymath}
  whereas $\tan\left(\theta_n(c)\right)=0$ due to $\beta_2=0$ and
  the fact that the eigenfunction corresponding to the eigenvalue $\lambda_n$
  has exactly $n$ zeros in the open interval $(0,c)$ imply that
  \begin{displaymath}
    \theta_n(c)=\pi+n\pi\;.
  \end{displaymath}
  A comparison with~(\ref{eq:thetaandR}) again yields the claimed
  result. Likewise, if $\alpha_2=0$ and $\beta_2\not=0$, we can
  choose the modified phase such that, as $n\to\infty$,
  \begin{displaymath}
    \theta_n(0)=0
    \quad\text{and}\quad
    \theta_n(c)=
    \frac{\pi}{2}+n\pi+O\left(\frac{1}{\sqrt{\lambda_n}}\right)\;,
  \end{displaymath}
  and, if $\alpha_2=\beta_2=0$, such that
  \begin{displaymath}
    \theta_n(0)=0
    \quad\text{and}\quad
    \theta_n(c)=(n+1)\pi\;,
  \end{displaymath}
  which, as above, imply the stated asymptotic formulae.
\end{proof}
We are now in a position to deduce the asymptotic formula
for the normalised eigenfunction $u_n$ as $n\to\infty$. We see that
the normalised eigenfunctions
of a Sturm--Liouville problem in Liouville normal form
are asymptotically close to the normalised
eigenfunctions of the Sturm--Liouville problem $u''=-\lambda u$
subject to
suitable boundary conditions.
\begin{propn}\label{propn:efuns}
  The normalised eigenfunctions $\{u_n:n\in\N_0\}$
  of a regular Sturm--Liouville problem in Liouville normal form
  whose corresponding eigenvalues $\{\lambda_n:n\in\N_0\}$
  form a strictly increasing sequence
  satisfy, for $t\in[0,c]$ and as $n\to\infty$,
  \begin{equation*}
    u_n(t)+O\left(\frac{1}{n}\right)=
    \begin{cases}
      \sqrt{\frac{2}{c}}\cos\left(\frac{n\pi t}{c}\right)
      & \text{if }\alpha_2,\beta_2\not= 0\\
      \sqrt{\frac{2}{c}}
      \cos\left(\left(n+\frac{1}{2}\right)\frac{\pi t}{c}\right)
      &\text{if }\alpha_2\not= 0\text{ and }\beta_2=0\\
      \sqrt{\frac{2}{c}}
      \sin\left(\left(n+\frac{1}{2}\right)\frac{\pi t}{c}\right)
      &\text{if }\alpha_2=0\text{ and }\beta_2\not=0\\
      \sqrt{\frac{2}{c}}\sin\left(\frac{\left(n+1\right)\pi t}{c}\right)
      & \text{if }\alpha_2=\beta_2=0
    \end{cases}\;.
  \end{equation*}
\end{propn}
\begin{proof}
  From the modified Pr\"ufer system
  and~(\ref{eq:thetaandR}), we obtain, as $\lambda\to\infty$,
  \begin{displaymath}
    \theta'(t)=\sqrt{\lambda}+O\left(\frac{1}{\sqrt{\lambda}}\right)
    \quad\text{and}\quad
    \int_{\theta(0)}^{\theta(c)}
    \left(\sin(\theta)\right)^2\dd\theta
    =\left[\frac{\theta}{2}-
      \frac{\sin\left(2\theta\right)}{4}\right]_{\theta(0)}^{\theta(c)}
    =\frac{\sqrt{\lambda}c}{2}+O(1)\;.
  \end{displaymath}
  It follows that
  \begin{displaymath}
    \int_0^c\left(\sin\left(\theta(t)\right)\right)^2\dd t
    =\left(\frac{1}{\sqrt{\lambda}}
      +O\left(\frac{1}{\sqrt{\lambda^3}}\right)\right)
    \int_{\theta(0)}^{\theta(c)}\left(\sin(\theta)\right)^2\dd\theta
    =\frac{c}{2}+O\left(\frac{1}{\sqrt{\lambda}}\right)\;.
  \end{displaymath}
  Using this together with~(\ref{eq:prufersub}) and
  (\ref{eq:thetaandR}) yields
  \begin{displaymath}
    \int_0^c\left(u(t)\right)^2\dd t
    =\left(R(0)+O\left(\frac{1}{\lambda}\right)\right)^2
    \left(\frac{c}{2\sqrt{\lambda}}
      +O\left(\frac{1}{\lambda}\right)\right)\;.
  \end{displaymath}
  Thus, the modified amplitude of a normalised eigenfunction
  satisfies, as $\lambda\to\infty$,
  \begin{equation}\label{eq:R0}
    R(0)=\sqrt{\frac{2\sqrt{\lambda}}{c}}
    \left(1+O\left(\frac{1}{\sqrt{\lambda}}\right)\right)\;.
  \end{equation}
  For the modified phase $\theta_n$
  which corresponds to the normalised eigenfunction with
  eigenvalue $\lambda_n$
  as in the proof of Proposition~\ref{propn:evals}, by
  using~(\ref{eq:thetaandR}), we find that, as $n\to\infty$, if
  $\alpha_2\not= 0$,
  \begin{displaymath}
    \sin\left(\theta_n(t)\right)=
    \sin\left(\frac{\pi}{2}+\sqrt{\lambda_n}t
      +O\left(\frac{1}{\sqrt{\lambda_n}}\right)\right)
    =\cos\left(\sqrt{\lambda_n}t
      +O\left(\frac{1}{\sqrt{\lambda_n}}\right)\right)\;,
  \end{displaymath}
  whereas, if $\alpha_2=0$,
  \begin{displaymath}
    \sin\left(\theta_n(t)\right)=
    \sin\left(\sqrt{\lambda_n}t
      +O\left(\frac{1}{\sqrt{\lambda_n}}\right)\right)\;.
  \end{displaymath}
  The claimed result follows from~(\ref{eq:prufersub}),
  (\ref{eq:thetaandR}), (\ref{eq:R0}) and after
  applying the asymptotic formula
  for the eigenvalues from Proposition~\ref{propn:evals} as well as
  the mean value theorem.
\end{proof}

\subsection{Asymptotic error in the eigenfunction expansion}
\label{sec:main}

Using the Liouville transformation and the asymptotic formulae for the
eigenfunctions and eigenvalues of Sturm--Liouville problems in
Liouville normal form, we prove Theorem~\ref{thm:regularSL} by
reducing the analysis to four base cases, 
which take care of different types of separated homogeneous boundary
conditions. One base case is covered
by~\cite[Theorem~1.2]{foster_habermann}, another one by
Proposition~\ref{propn:baseDN} below, and the remaining two cases can
be deduced from these results.

When proving Proposition~\ref{propn:baseDN}, we employ a similar proof
strategy as was used for~\cite[Theorem~1.2]{foster_habermann} and
\cite[Theorem~1.5]{semicircle}, that is, we split the
analysis into an on-diagonal and an off-diagonal part, with the pointwise
convergence on the diagonal being a consequence of a convergence of
moments and
two local uniform bounds, which allows for an
application of the Arzel\`{a}--Ascoli theorem.
The main difference is that we do not
compute the moments on the diagonal explicitly, and instead illustrate
a powerful approach exploiting the
Sturm--Liouville differential equation
which we apply in our analysis for classical
orthogonal polynomial systems in Section~\ref{sec:COP}.
\begin{propn}\label{propn:baseDN}
  We have, for $s,t\in[0,1]$,
  \begin{displaymath}
    \lim_{N\to\infty}N
    \sum_{n=N+1}^\infty
    \frac{2\sin\left(\left(n+\frac{1}{2}\right)\pi s\right)
      \sin\left(\left(n+\frac{1}{2}\right)\pi t\right)}
      {\left(n+\frac{1}{2}\right)^2\pi^2}=
    \begin{cases}
      \dfrac{1}{\pi^2} & \text{if } s=t\text{ and } t\in(0,1)\\[0.7em]
      \dfrac{2}{\pi^2} & \text{if } s=t=1\\[0.7em]
      0 & \text{otherwise}
    \end{cases}\;.
  \end{displaymath}
\end{propn}
\begin{proof}
  We start by proving the claimed convergence on the diagonal
  $\{s=t\}$ and then 
  deduce the off-diagonal convergence from the on-diagonal
  convergence.
  For $n\in\N_0$, the function
  $u_n\colon[0,1]\to\R$ given by
  \begin{displaymath}
    u_n(t)=\sqrt{2}\sin\left(\left(n+\frac{1}{2}\right)\pi t\right)
  \end{displaymath}
  is a normalised eigenfunction of the Sturm--Liouville problem
  \begin{equation}\label{eq:SLP_trig2}
    \frac{\db^2 u(t)}{\db t^2}
    =-\lambda u(t)
    \quad\text{subject to}\quad
    u(0)=0\;,\quad u'(1)=0\;,
  \end{equation}
  with corresponding eigenvalue
  \begin{displaymath}
    \lambda_n=\left(n+\frac{1}{2}\right)^2\pi^2\;.
  \end{displaymath}
  Hence, we are interested in determining the pointwise limit of
  the functions $S_N\colon [0,1]\to\R$ defined by, for $N\in\N_0$,
  \begin{displaymath}
    S_N(t)=N\sum_{n=N+1}^\infty\frac{\left(u_n(t)\right)^2}
    {\lambda_n}
  \end{displaymath}
  as $N\to\infty$. Due to $u_n(0)=0$ for all $n\in\N_0$ we have
  $\lim_{N\to\infty}S_N(0)=0$, as claimed. For $t=1$, we need
  to compute 
  \begin{displaymath}
    \lim_{N\to\infty}N
    \sum_{n=N+1}^\infty
    \frac{2}{\left(n+\frac{1}{2}\right)^2\pi^2}\;,
  \end{displaymath}
  which is indeed $\frac{2}{\pi^2}$ because
  \begin{equation}\label{eq:F_unifbound1}
    \lim_{N\to\infty}N \sum_{n=N+1}^\infty
    \frac{1}{\left(n+\frac{1}{2}\right)^2}
    \leq \lim_{N\to\infty}N \sum_{n=N+1}^\infty
    \left(\frac{1}{n-\frac{1}{2}}-\frac{1}{n+\frac{1}{2}}\right)
    =\lim_{N\to\infty}\frac{N}{N+\frac{1}{2}}=1
  \end{equation}
  and
  \begin{equation}\label{eq:F_unifbound2}
    \lim_{N\to\infty}N \sum_{n=N+1}^\infty
    \frac{1}{\left(n+\frac{1}{2}\right)^2}
    \geq \lim_{N\to\infty}N \sum_{n=N+1}^\infty
    \left(\frac{1}{n+\frac{1}{2}}-\frac{1}{n+\frac{3}{2}}\right)
    =\lim_{N\to\infty}\frac{N}{N+\frac{3}{2}}=1\;.
  \end{equation}
  We further observe that the bound which was used to
  deduce~(\ref{eq:F_unifbound1}) shows that, for all
  $N\in\N_0$ and for all $t\in[0,1]$,
  \begin{equation}\label{eq:F_bound}
    \left|S_N(t)\right|\leq \frac{2}{\pi^2}\;.
  \end{equation}
  Since the Green's function $G\colon[0,1]\times[0,1]\to\R$ of
  the Sturm--Liouville problem~(\ref{eq:SLP_trig2}) is given by
  \begin{equation}\label{eq:Fgreen}
    G(s,t)=\min(s,t)\;,
  \end{equation}
  we obtain
  \begin{displaymath}
    S_N(t)=N\left(t-\sum_{n=0}^N\frac{\left(u_n(t)\right)^2}
      {\lambda_n}\right)\;.
  \end{displaymath}
  It follows that
  \begin{displaymath}
    S_N'(t)=N\left(1-\sum_{n=0}^N\frac{2u_n(t)u_n'(t)}{\lambda_n}\right)
    =N\left(1-
      \sum_{n=0}^N\frac{4\sin\left((2n+1\right)\pi t)}{(2n+1)\pi}\right)\;.
  \end{displaymath}
  Similar to the proof of~\cite[Lemma~4.2]{foster_habermann}, we use
  \begin{displaymath}
    \sum_{n=0}^N\cos\left((2n+1)\pi t\right)
    =\frac{\sin\left(2(N+1)\pi t\right)}{2\sin(\pi t)}
  \end{displaymath}
  to rewrite, for $t\in(0,1)$,
  \begin{displaymath}
    \sum_{n=0}^N\left(\frac{(-1)^n}{(2n+1)\pi}
      -\frac{\sin\left((2n+1\right)\pi t)}{(2n+1)\pi}\right)
    =-\sum_{n=0}^N\int_{\frac{1}{2}}^t
    \cos\left((2n+1)\pi r\right)\dd r
    =-\int_{\frac{1}{2}}^t
    \frac{\sin\left(2(N+1)\pi r\right)}{2\sin(\pi r)}\dd r\;.
  \end{displaymath}
  As in the proof of~\cite[Lemma~4.2]{foster_habermann}, integration
  by parts shows that, for all $\eps>0$, the family
  \begin{displaymath}
    \left\{N\int_{\frac{1}{2}}^t
      \frac{\sin\left(2(N+1)\pi r\right)}{2\sin(\pi r)}\dd r
      :N\in\N_0\text{ and }t\in[\eps,1-\eps]\right\}
  \end{displaymath}
  is uniformly bounded. Since the Leibniz series
  $\sum_{n=0}^\infty\frac{(-1)^n}{2n+1}$ takes the value
  $\frac{\pi}{4}$, we further have
  \begin{displaymath}
    N\left(\frac{1}{4}-\sum_{n=0}^N\frac{(-1)^n}{(2n+1)\pi}\right)
    =N\sum_{n=N+1}^\infty\frac{(-1)^n}{(2n+1)\pi}\;,
  \end{displaymath}
  which, by a similar telescoping argument as above,
  is bounded uniformly in $N\in\N_0$.
  It follows that the derivative $S_N'$ is locally uniformly bounded on
  the open interval $(0,1)$.
  This together with~(\ref{eq:F_bound})
  implies that the Arzel\`{a}--Ascoli theorem can be applied
  locally to any
  subsequence of $(S_N)_{N\geq 0}$. Repeatedly appealing to the
  Arzel\`{a}--Ascoli theorem and employing a diagonal argument, we deduce
  there exists a subsequence of $(S_N)_{N\geq 0}$ which converges
  pointwise to a continuous limit function on $(0,1)$. To show that
  the full sequence converges pointwise
  and to identify the limit function, we make use of a moment
  argument.
  This is where we exploit the
  Sturm--Liouville differential equation.
  Integration by parts and~(\ref{eq:SLP_trig2}) imply that,
  for all $n,k\in\N_0$,
  \begin{displaymath}
    \lambda_n\int_0^1t^k\left(u_n(t)\right)^2\dd t
    =-\int_0^1t^ku_n(t)u_n''(t)\dd t
    =\int_0^1t^k\left(u_n'(t)\right)^2\dd t
    +\int_0^1 kt^{k-1}u_n(t)u_n'(t)\dd t\;.
  \end{displaymath}
  On the other hand, we obtain from~(\ref{eq:SLP_trig2}) that
  \begin{displaymath}
    \frac{\db}{\db t}\left(
      t^{k+1}\left(u_n'(t)\right)^2
      +\lambda_nt^{k+1}\left(u_n(t)\right)^2\right)
    =(k+1)t^k\left(u_n'(t)\right)^2+
    (k+1)\lambda_nt^k\left(u_n(t)\right)^2\;,
  \end{displaymath}
  which implies
  \begin{displaymath}
    \int_0^1t^k\left(u_n'(t)\right)^2\dd t
    +\lambda_n\int_0^1t^k\left(u_n(t)\right)^2\dd t
    =\frac{\lambda_n\left(u_n(1)\right)^2}{k+1}
    =\frac{2\lambda_n}{k+1}\;.
  \end{displaymath}
  It suffices to observe that, for $k\in\N_0$ fixed and as $n\to\infty$,
  \begin{displaymath}
    \int_0^1 kt^{k-1}u_n(t)u_n'(t)\dd t
    =\frac{1}{2}\left(
      2k-\int_0^1 k(k-1)t^{k-2}\left(u_n(t)\right)^2\dd t
    \right)
    =O(1)\;,
  \end{displaymath}
  which can be seen, for instance, from
  \begin{displaymath}
    \int_0^1 k(k-1)t^{k-2}\left(u_n(t)\right)^2\dd t
    \leq k(k-1)\int_0^1 \left(u_n(t)\right)^2\dd t
    =k(k-1)\;,
  \end{displaymath}
  to deduce that, for $k\in\N_0$ fixed and as $n\to\infty$,
  \begin{displaymath}
    2\lambda_n\int_0^1t^k\left(u_n(t)\right)^2\dd t
    =\frac{2\lambda_n}{k+1}+O(1)\;.
  \end{displaymath}
  Since $\lambda_n=O(n^2)$ as $n\to\infty$, this yields
  \begin{displaymath}
    \int_0^1\frac{t^k\left(u_n(t)\right)^2}{\lambda_n}\dd t
    =\frac{1}{(k+1)\lambda_n}+O\left(n^{-4}\right)\;.
  \end{displaymath}
  Using Fubini's theorem to interchange integration and summation, the
  expression $\lambda_n=\left(n+\frac{1}{2}\right)^2\pi^2$ and the
  bound, for $N\in\N$,
  \begin{displaymath}
    \int_{N+1}^\infty\frac{1}{r^4}\dd r
    \leq\sum_{n=N+1}^\infty\frac{1}{n^4}
    \leq\frac{1}{N^4}+\int_{N+1}^\infty\frac{1}{r^4}\dd r
  \end{displaymath}
  as well as~(\ref{eq:F_unifbound1}) and (\ref{eq:F_unifbound2}), we
  conclude, for all $k\in\N_0$,
  \begin{displaymath}
    \lim_{N\to\infty}\int_0^1 t^k S_N(t)\dd t
    =\frac{1}{(k+1)\pi^2}
    =\int_0^1\frac{t^k}{\pi^2}\dd t\;.
  \end{displaymath}
  If the sequence $(S_N)_{N\geq 0}$ did not converge pointwise, we
  could again apply the Arzel\`{a}--Ascoli theorem and a diagonal argument to
  extract a second subsequence of $(S_N)_{N\geq 0}$ which converged
  pointwise but to a different continuous limit function on $(0,1)$
  compared to the first subsequence. As this contradicts the convergence of
  moments, we obtain that, for $t\in(0,1)$,
  \begin{displaymath}
    \lim_{N\to\infty}S_N(t)=\frac{1}{\pi^2}\;.
  \end{displaymath}

  It remains to prove the pointwise convergence away from the
  diagonal. As for~\cite[Theorem~1.2]{foster_habermann} this follows
  from the on-diagonal convergence. For $t\in(0,1)$, we
  see that
  \begin{displaymath}
    \lim_{N\to\infty}N\sum_{n=N+1}^\infty
    \frac{\cos\left(2\left(n+\frac{1}{2}\right)\pi t\right)}
    {\left(n+\frac{1}{2}\right)^2\pi^2}
    =\lim_{N\to\infty}N\sum_{n=N+1}^\infty
      \frac{1-2\left(\sin\left(\left(n+\frac{1}{2}\right)\pi
            t\right)\right)^2}
      {\left(n+\frac{1}{2}\right)^2\pi^2}
    =\frac{1}{\pi^2}-\frac{1}{\pi^2}=0\;.
  \end{displaymath}
  Applying the identity, for $n\in\N_0$,
  \begin{displaymath}
    2\sin\left(\left(n+\frac{1}{2}\right)\pi s\right)
    \sin\left(\left(n+\frac{1}{2}\right)\pi t\right)
    =\cos\left(\left(n+\frac{1}{2}\right)\pi(t-s)\right)
    -\cos\left(\left(n+\frac{1}{2}\right)\pi(t+s)\right)\;,
  \end{displaymath}
  we deduce the claimed convergence to zero for $s\not= t$.
\end{proof}
Using characteristic functions as in the proof
of~\cite[Theorem~1.6]{semicircle}, Theorem~\ref{thm:KLBM} is a
consequence of Proposition~\ref{propn:baseDN} since the fluctuation
processes $(F_t^N)_{t\in[0,1]}$ are zero-mean Gaussian processes whose
covariance functions, due to~(\ref{eq:Fgreen}), are exactly given by
\begin{displaymath}
  N\sum_{n=N+1}^\infty
  \frac{2\sin\left(\left(n+\frac{1}{2}\right)\pi s\right)
    \sin\left(\left(n+\frac{1}{2}\right)\pi t\right)}
  {\left(n+\frac{1}{2}\right)^2\pi^2}\;.
\end{displaymath}
Before we turn to the proof of Theorem~\ref{thm:regularSL}, we show
how~Proposition~\ref{propn:baseDN}
and~\cite[Theorem~1.2]{foster_habermann}, whose result
is~(\ref{eq:baseDD}), allow us to
analyse the remaining two base cases.
From~\cite[Theorem~1.2]{foster_habermann}, it follows that, for
$t\in(0,1)$,
\begin{displaymath}
  \lim_{N\to\infty}N \sum_{n=N+1}^\infty
  \frac{\cos\left(2n\pi t\right)}{n^2\pi^2}=0\;,
\end{displaymath}
see~\cite[Corollary~1.3]{foster_habermann}. This together
with~\cite[Theorem~1.2]{foster_habermann} and
\begin{displaymath}
  2\cos(n\pi s)\cos(n\pi t)
  =\cos\left(n\pi(t-s)\right)+\cos\left(n\pi(t+s)\right)
\end{displaymath}
as well as
\begin{displaymath}
  N\sum_{n=N+1}^\infty\frac{2\left(\cos(n\pi t)\right)^2}{n^2\pi^2}
  =N\left(\sum_{n=N+1}^\infty\frac{2}{n^2\pi^2}
  -\sum_{n=N+1}^\infty\frac{2\left(\sin(n\pi t)\right)^2}{n^2\pi^2}\right)
\end{displaymath}
implies that, for $s,t\in[0,1]$,
\begin{equation}\label{eq:baseNN}
  \lim_{N\to\infty}N
  \sum_{n=N+1}^\infty\frac{2\cos(n\pi s)\cos(n\pi t)}{n^2\pi^2}=
  \begin{cases}
    \dfrac{1}{\pi^2} & \text{if } s=t\text{ and } t\in(0,1)\\[0.7em]
    \dfrac{2}{\pi^2} & \text{if } s=t=0\text{ or } s=t=1\\[0.7em]
    0 & \text{otherwise}
  \end{cases}\;.
\end{equation}
Similarly, we conclude, for $s,t\in[0,1]$,
\begin{equation}\label{eq:baseND}
  \lim_{N\to\infty}N
  \sum_{n=N+1}^\infty
  \frac{2\cos\left(\left(n+\frac{1}{2}\right)\pi s\right)
    \cos\left(\left(n+\frac{1}{2}\right)\pi t\right)}
  {\left(n+\frac{1}{2}\right)^2\pi^2}=
  \begin{cases}
    \dfrac{1}{\pi^2} & \text{if } s=t\text{ and } t\in(0,1)\\[0.7em]
    \dfrac{2}{\pi^2} & \text{if } s=t=0\\[0.7em]
    0 & \text{otherwise}
  \end{cases}\;.
\end{equation}
\begin{proof}[Proof of Theorem~\ref{thm:regularSL}]
  Since $p,w\in C^2([a,b])$, we can apply the Liouville transformation
  to the given regular Sturm--Liouville problem. Under this
  transformation, the family $\{\phi_n:n\in\N_0\}$ of
  eigenfunctions on $[a,b]$
  is transformed to the  family $\{u_n:n\in\N_0\}$ of eigenfunctions on $[0,c]$
  defined by
  \begin{displaymath}
    u_n(t)=\sqrt[4]{p(x)w(x)}\phi_n(x)\;,
  \end{displaymath}
  where
  \begin{displaymath}
    t=\int_a^x\sqrt{\frac{w(z)}{p(z)}}\dd z
    \quad\text{and}\quad
    c=\int_a^b\sqrt{\frac{w(z)}{p(z)}}\dd z\;.
  \end{displaymath}
  Further setting
  \begin{displaymath}
    s=\int_a^y\sqrt{\frac{w(z)}{p(z)}}\dd z
  \end{displaymath}
  and observing that
  \begin{displaymath}
    \int_a^b\left(\phi_n(z)\right)^2w(z)\dd z
    =\int_0^c\left(u_n(r)\right)^2\dd r\;,
  \end{displaymath}
  we can write
  \begin{equation}\label{eq:transformation}
    N\sum_{n=N+1}^\infty \frac{\phi_n(x)\phi_n(y)}
    {\lambda_n\int_a^b\left(\phi_n(z)\right)^2w(z)\dd z}
    =N\sum_{n=N+1}^\infty\frac{1}{\sqrt[4]{p(x)w(x)}\sqrt[4]{p(y)w(y)}}
    \frac{u_n(s)u_n(t)}
    {\lambda_n\int_0^c\left(u_n(r)\right)^2\dd r}\;.
  \end{equation}
  Moreover, since $p$ and $w$ are positive functions by assumption,
  the four cases distinguished by $\alpha_2=0$ or
  $\alpha_2\not=0$, and $\beta_2=0$ or $\beta_2\not=0$ remain
  invariant under the Liouville transformation, whilst the values of the
  non-zero constants might change.
  In particular, if we have $\alpha_2,\beta_2\not=0$ in the original
  Sturm--Liouville problem then, by
  Proposition~\ref{propn:efuns}, we obtain, as $N\to\infty$,
  \begin{displaymath}
    N\sum_{n=N+1}^\infty
    \frac{u_n(s)u_n(t)}
    {\lambda_n\int_0^c\left(u_n(r)\right)^2\dd r}
    =N\sum_{n=N+1}^\infty
      \frac{2\cos\left(\frac{n\pi s}{c}\right)
        \cos\left(\frac{n\pi t}{c}\right)
        +O\left(\frac{1}{n}\right)}{c\lambda_n}\;.
  \end{displaymath}
  Applying Proposition~\ref{propn:evals}, which for
  $\alpha_2,\beta_2\not=0$ gives that, as $n\to\infty$,
  \begin{displaymath}
    c^2\lambda_n=n^2\pi^2+O(1)\;,
  \end{displaymath}
  we conclude
  \begin{equation}\label{eq:conv_main}
    N\sum_{n=N+1}^\infty
    \frac{u_n(s)u_n(t)}
    {\lambda_n\int_0^c\left(u_n(r)\right)^2\dd r}
    =N\sum_{n=N+1}^\infty\left(
      \frac{2c\cos\left(\frac{n\pi s}{c}\right)
        \cos\left(\frac{n\pi t}{c}\right)}{n^2\pi^2}
      +O\left(\frac{1}{n^3}\right)\right)\;.
  \end{equation}
  It suffices to note that
  \begin{displaymath}
    0\leq \lim_{N\to\infty} N\sum_{n=N+1}^\infty\frac{1}{n^3}
    \leq \lim_{N\to\infty} \sum_{n=N+1}^\infty\frac{1}{n^2}=0
  \end{displaymath}
  to deduce the claimed result from~(\ref{eq:transformation}),
  (\ref{eq:conv_main}) and (\ref{eq:baseNN}). The remaining
  three cases given in terms of the separated homogeneous boundary
  conditions follow similarly from Proposition~\ref{propn:evals} and
  Proposition~\ref{propn:efuns} as well
  as from~(\ref{eq:baseND}), Proposition~\ref{propn:baseDN} and
  \cite[Theorem~1.2]{foster_habermann}, respectively.
\end{proof}

\subsection{Illustrating example}
\label{sec:1example}

We present one example to illustrate Theorem~\ref{thm:regularSL}
which goes beyond the four base cases encountered above but where the
eigenfunctions and eigenvalues can still be computed explicitly,
allowing for the inclusion of a plot.

\begin{eg0}\rm
  We consider the eigenvalue problem
  \begin{align}
    \phi''(x)+3\phi'(x)+2\phi(x)=-\lambda \phi(x)
    \quad&\text{subject to}\nonumber
    \quad \phi(0)=0\;,\quad \phi(1)=0\;,
    \intertext{which can be rewritten as the Sturm--Liouville problem}
    \left(\e^{3x}\phi'(x)\right)'+2\e^{3x}\phi(x)=-\lambda \e^{3x}\phi(x)
    \quad&\text{subject to}\label{eq:SLex1}
    \quad \phi(0)=0\;,\quad \phi(1)=0\;.
  \end{align}
  Thus, we are in the setting of Theorem~\ref{thm:regularSL} with
  $a=0$, $b=1$ and, for $x\in[0,1]$,
  \begin{displaymath}
    p(x)=w(x)=\e^{3x}\;.
  \end{displaymath}
  A direct computation shows that the normalised eigenfunctions are
  given by, for $n\in\N_0$,
  \begin{displaymath}
    \phi_n(x)=\sqrt{2}\e^{-3x/2}\sin\left((n+1)\pi x\right)\;,
  \end{displaymath}
  with corresponding eigenvalues
  \begin{displaymath}
    \lambda_n=\frac{1+4(n+1)^2\pi^2}{4}\;.
  \end{displaymath}
  To construct the Green's function $G\colon[0,1]\times[0,1]\to\R$ of
  the Sturm--Liouville
  problem~(\ref{eq:SLex1}), we use solutions to the homogeneous
  differential equation
  \begin{displaymath}
    v''(x)+3v'(x)+2v(x)=0
  \end{displaymath}
  one subject to $v(0)=0$ and another one subject to $v(1)=0$. Taking
  \begin{displaymath}
    v_1(x)=\e^{-x}-\e^{-2x}
    \quad\text{as well as}\quad
    v_2(x)=\e^{-2x}-\e^{-1-x}
  \end{displaymath}
  and computing
  \begin{displaymath}
    p(x)\left(v_1'(x)v_2(x)-v_1(x)v_2'(x)\right)
    =1-\e^{-1}\;,
  \end{displaymath}
  we obtain, for $x,y\in[0,1]$,
  \begin{displaymath}
    G(x,y)=
    \begin{cases}
      \dfrac{\left(\e^{-x}-\e^{-2x}\right)
        \left(\e^{-2y}-\e^{-1-y}\right)}{1-\e^{-1}} & \text{if }
      0\leq x\leq y\\[0.7em]
      \dfrac{\left(\e^{-y}-\e^{-2y}\right)
        \left(\e^{-2x}-\e^{-1-x}\right)}{1-\e^{-1}} & \text{if }
      y< x\leq 1
    \end{cases}\;.
  \end{displaymath}
  In particular, we have, for $N\in\N_0$,
  \begin{displaymath}
    N\sum_{n=N+1}^\infty \frac{\left(\phi_n(x)\right)^2}
    {\lambda_n}
    =N\left(G(x,x)-
      \sum_{n=0}^N \frac{\left(\phi_n(x)\right)^2}{\lambda_n}\right)\;.
  \end{displaymath}
  A plot of this function for $N=100$ is shown in
  Figure~\ref{fig:example}, which nicely illustrates that,
  on the diagonal away from its endpoints,
  the rescaled
  error in approximating the Green's function is close to
  be given by
  $x\mapsto \e^{-3x}/\pi^2$, as asserted by Theorem~\ref{thm:regularSL}.
  \begin{figure}[ht]
    \centering
    \includegraphics[width=0.6\linewidth]{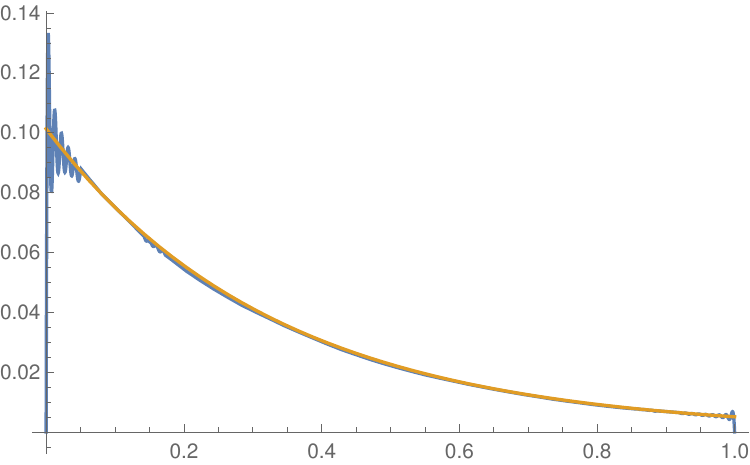}
    \caption{Rescaled error in approximating Green's function for
      $N=100$ (blue) and $x\mapsto \e^{-3x}/\pi^2$ (yellow) on $(0,1)$.}
    \label{fig:example}
  \end{figure}
\end{eg0}

\section{Green's function formed by classical
  orthogonal polynomials}
\label{sec:COP}

We prove Theorem~\ref{thm:hermite}, Theorem~\ref{thm:laguerre} and
Theorem~\ref{thm:jacobi}, that is, we derive the asymptotic error in
the eigenfunction expansion for the Green's function associated with
the Hermite polynomials, the associated Laguerre polynomials and the
Jacobi polynomials.
After we first show that the limit moments on the
diagonal satisfy the desired recurrence relation and find a
Christoffel--Darboux type formula for any family of classical
orthogonal polynomials, we then need to consider the Hermite polynomials,
the associated Laguerre polynomials and the Jacobi polynomials
separately to conclude the proofs.
This involves using asymptotic
formulae for the orthogonal polynomials to obtain sufficiently strong
bounds and, for
the Hermite polynomials and the associated Laguerre polynomials,
showing that the
moments on the diagonal satisfy the necessary initial condition.
For the Jacobi polynomials, the latter seems challenging which is why 
we employ a workaround.

For a family $\{Y_n:n\in\N_0\}$ of
classical orthogonal polynomials on
the interval $I$ of orthogonality, there exists a linear function $L$
and a polynomial $Q$ of degree at most two as well as a family
$\{\lambda_n: n\in\N_0\}$ such that, for all $n\in\N_0$ and
for all $x\in I$,
\begin{equation}\label{COPDE}
  Q(x)Y_n''(x)+L(x)Y_n'(x)+\lambda_n Y_n(x)=0\;.
\end{equation}
They arise by means of linear transformation from the Jacobi polynomials
if $Q$ is of degree two and has two distinct zeros, from the associated
Laguerre polynomials if $Q$
is linear, and from the Hermite polynomials if $Q$ is constant. In
particular, the three main theorems proved in this section allow us to
deduce the considered asymptotic error for other classical orthogonal
polynomial systems.

Throughout, we assume that the degree of a polynomial is given by its index.
Using the
integrating factor
\begin{displaymath}
  R(x)=\exp\left(\int^x\frac{L(z)}{Q(z)}\dd z\right)\;,
\end{displaymath}
we can write~(\ref{COPDE}) in its standard Sturm--Liouville form
\begin{displaymath}
  \left(P(x)Y_n'(x)\right)'=-\lambda_nW(x) Y_n(x)
  \quad\text{with}\quad
  P(x)=R(x)\quad\text{and}\quad W(x)=\frac{R(x)}{Q(x)}\;,
\end{displaymath}
where $\{\lambda_n: n\in\N_0\}$ is the family of eigenvalues
corresponding to $\{Y_n: n\in\N_0\}$. As discussed in
Lesky~\cite[Satz~3]{lesky}, the eigenvalues are mutually distinct and
given by
\begin{equation}\label{eq:evals}
  \lambda_n=-n\left(\frac{n-1}{2}Q''+L'\right)\;,
\end{equation}
which is indeed a constant. Moreover, it follows that
$\lambda_n\not=0$ for all $n\in\N$ since $\lambda_0=0$. 
It is further shown, see Lesky~\cite[Satz~4]{lesky},
that at a boundary point $c\in\pt I$ of the interval of
orthogonality with respect to the weight function $W\colon I\to\R$, we
have
\begin{displaymath}
  P(c)=R(c)=0
  \quad\text{if }c\in\R
\end{displaymath}
and
\begin{displaymath}
  \lim_{x\to c}x^kW(x)=\lim_{x\to c}\frac{x^kR(x)}{Q(x)}=0
  \quad\text{for all }k\in\N_0
  \quad\text{if }c=\pm\infty\;.
\end{displaymath}
We remark that for a given interval $I$ of orthogonality and a
suitable weight function $W$, we can recover the associated family
$\{Y_n:n\in\N_0\}$
of classical orthogonal polynomials by applying the
Gram--Schmidt process to the monomials $\{x^n:n\in\N_0\}$ on $I$ with
respect to the $L^2(I,W(x)\dd x)$ inner product and by imposing a
normalisation condition.

The one system of polynomials arising from Sturm--Liouville
differential equations which is notably missing in our
discussion above compared to the classification given in
Bochner~\cite{bochner} are the Bessel polynomials. As discussed by
Littlejohn and Krall in~\cite{littlekrall}, no sufficiently satisfying
weight function has been identified for the Bessel polynomials, apart
from the orthogonality relation obtained by integrating over the unit
circle in the complex plane with respect to $\e^{-2/x}$ given in Krall
and Frink~\cite{krall}. Recalling that the weight function features in
our expression for the asymptotic error,
one may wonder if one could use an error analysis of the eigenfunction
expansion for the Green's function associated with the Bessel polynomials
to find a good candidate for the long sought-after weight function.
However, initial plots suggest that this is not feasible.

For a given family $\{Y_n:n\in\N_0\}$ of
classical orthogonal polynomials on the interval $I$, let $M_n$ denote
the square of the
$L^2(I,W(x)\dd x)$ norm of $Y_n$, that is,
\begin{displaymath}
  M_n=\int_IW(x)\left(Y_n(x)\right)^2\dd x\;.
\end{displaymath}
Choosing $\tau\in\{\frac{1}{2},1\}$ such that $\sqrt{\lambda}_n=O(n^{\tau})$ as
$n\to\infty$, whose existence is guaranteed by~(\ref{eq:evals}) and the
result that the eigenvalues $\lambda_n$ are mutually distinct, we are
then interested in studying, 
for $x,y\in I\setminus\pt I$,
\begin{equation}\label{eq:green_fluct_COP}
  \lim_{N\to\infty}N^{\tau}\sum_{n=N+1}^\infty
  \frac{Y_n(x)Y_n(y)}{M_n\lambda_n}\;,
\end{equation}
with the one caveat that in the off-diagonal convergence for the
associated Laguerre polynomials we can only deal with scaling
exponents strictly smaller than $\frac{1}{2}$.

In our analysis, we use the same approach which
already proved powerful in~\cite{foster_habermann, semicircle} and
which was demonstrated in the previous section, that
is, we split our considerations into an on-diagonal part, consisting
of a moment argument and local uniform
bounds feeding into the Arzel\`{a}--Ascoli theorem, and an
off-diagonal part. While we are able to
establish in general that, under suitable
boundary conditions, in the limit as $N\to\infty$ the moments on
the diagonal satisfy the same recurrence relation as the desired limit
function and to derive a general Christoffel--Darboux type formula
used
for the off-diagonal convergence,
it seems difficult to develop a full general analysis.
In particular, to show that the moments
satisfy the desired initial condition, to obtain the local uniform
bounds and to complete the analysis away from the diagonal, we need to
distinguish between the Hermite polynomials, the associated Laguerre
polynomials and the Jacobi polynomials. We further rely
on existing asymptotic expansions for these polynomials as given in
Szeg\H{o}~\cite[Chapter~8]{szego}.

We start by determining
the recurrence relation satisfied, in the limit
as $N\to\infty$, by the moments
of the function~(\ref{eq:green_fluct_COP}) restricted to the diagonal
$\{x=y\}$. As the final part of the moment
analysis requires us to consider the three main families of classical
orthogonal polynomials separately, we delay the observation that the
integrals and series needed
in the following are well-defined and
finite until
then.
Moreover, while the boundary conditions and the asymptotic assumptions in the
propositions below are satisfied whenever we want to apply the
propositions, we still include the assumptions explicitly to make it clear
what is needed.
\begin{propn}\label{propn:moments}
  Let $\{Y_n: n\in\N_0\}$ be a family of classical orthogonal
  polynomials which are orthogonal on the interval $I=(a,b)$ with
  respect to the weight function $W\colon I\to\R$ and which together
  with the family $\{\lambda_n: n\in\N_0\}$ of eigenvalues solve the
  Sturm--Liouville differential equation
  \begin{displaymath}
    (P(x)Y_n'(x))'=-\lambda_nW(x) Y_n(x)\;,
  \end{displaymath}
  or equivalently
  \begin{displaymath}
    Q(x)Y_n''(x)+L(x)Y_n'(x)+\lambda_n Y_n(x)=0\;.
  \end{displaymath}
  Assume that, for all $l,n\in\N_0$ and for $c\in\{a,b\}$,
  \begin{equation}\label{eq:bc_COP}
    \lim_{x\to c}x^{l+1}P(x)W(x)\left(Y_n(x)\right)^2=
    \lim_{x\to c}x^lP(x)Y_n(x)=\lim_{x\to c}x^lP(x)Y_n'(x)=0
  \end{equation}
  and that, for $l\in\N_0$ fixed and as $n\to\infty$, 
  \begin{equation}\label{eq:ac_COP}
    \int_I\left(x^l P(x)\right)'P(x)Y_n(x)Y'_n(x)\dd x
    =O\left(M_n\right)\;.
  \end{equation}
  Then the limit moments, for $k\in\N_0$,
  \begin{displaymath}
    m_{k}=\lim_{N\to\infty}N^\tau \sum_{n=N+1}^\infty\int_I
    \frac{x^k \left(W(x)\right)^2\left(Y_n(x)\right)^2}{M_n\lambda_n}
    \dd x
  \end{displaymath}
  satisfy the recurrence relation
  \begin{displaymath}
    (k+1)Q(0)m_{k}+
    \left(L(0)+\left(k+\frac{1}{2}\right)Q'(0)\right)m_{k+1}+
    \left(L'(0)+\frac{k}{2}Q''(0)\right)m_{k+2}=0\;.
  \end{displaymath}
\end{propn}
\begin{proof}
  Using the Sturm--Liouville differential equation, we obtain that, for
  all $k,n\in\N_0$,
  \begin{displaymath}
    \lambda_n\int_Ix^kP(x)W(x)\left(Y_n(x)\right)^2\dd x
    =-\int_I x^k(P(x)Y_n'(x))'P(x)Y_n(x)\dd x\;.
  \end{displaymath}
  After integrating by parts, we have
  \begin{align*}
    &\lambda_n\int_Ix^kP(x)W(x)\left(Y_n(x)\right)^2\dd x\\
    &\quad=-\left.x^k\left(P(x)\right)^2Y_n(x)Y_n'(x)\right|_a^b
      +\int_Ix^k\left(P(x) Y_n'(x)\right)^2\dd x
      +\int_I\left(x^kP(x)\right)'P(x)Y_n(x)Y_n'(x)\dd x\;.
  \end{align*}
  Due to the boundary conditions~(\ref{eq:bc_COP}), which give
  \begin{displaymath}
    \left.x^k\left(P(x)\right)^2Y_n(x)Y_n'(x)\right|_a^b=0\;,
  \end{displaymath}
  and the asymptotic
  assumption~(\ref{eq:ac_COP}), the above reduces to, for $k\in\N_0$
  fixed and as $n\to\infty$,
  \begin{equation}\label{eq:asymptotic}
    \lambda_n\int_Ix^kP(x)W(x)\left(Y_n(x)\right)^2\dd x
    =\int_Ix^k\left(P(x) Y_n'(x)\right)^2\dd x
    +O\left(M_n\right)\;.
  \end{equation}
  Integrating by parts yet again and using~(\ref{eq:bc_COP}), we
  further see that
  \begin{align*}
    \int_Ix^k\left(P(x) Y_n'(x)\right)^2\dd x
    &=-\frac{2}{k+1}\int_Ix^{k+1}P(x) Y_n'(x)
    \left(P(x) Y_n'(x)\right)'\dd x\\
    &=\frac{2\lambda_n}{k+1}\int_Ix^{k+1}P(x) W(x) Y_n(x) Y_n'(x)\dd x
  \end{align*}
  as well as
  \begin{align*}
    &\int_Ix^kP(x)W(x)\left(Y_n(x)\right)^2\dd x\\
    &\quad=-\frac{2}{k+1}\int_Ix^{k+1}P(x) W(x) Y_n(x) Y_n'(x)\dd x
    -\frac{1}{k+1}\int_Ix^{k+1}\left(P(x) W(x)\right)'
    \left(Y_n(x)\right)^2\dd x\;,
  \end{align*}
  which can be put together to give
  \begin{align*}
    &\int_Ix^k\left(P(x) Y_n'(x)\right)^2\dd x\\
    &\quad=-\lambda_n\int_Ix^kP(x)W(x)\left(Y_n(x)\right)^2\dd x
    -\frac{\lambda_n}{k+1}\int_Ix^{k+1}\left(P(x) W(x)\right)'
    \left(Y_n(x)\right)^2\dd x\;.
  \end{align*}
  Combining this with~(\ref{eq:asymptotic}) yields, for $k\in\N_0$
  fixed and as $n\to\infty$,
  \begin{equation}\label{eq:mom_con1}
    2\lambda_n\int_Ix^kP(x)W(x)\left(Y_n(x)\right)^2\dd x
    +\frac{\lambda_n}{k+1}\int_Ix^{k+1}\left(P(x) W(x)\right)'
    \left(Y_n(x)\right)^2\dd x=O(M_n)\;.
  \end{equation}
  Observing that
  \begin{displaymath}
    P(x)W(x)=Q(x)\left(\frac{R(x)}{Q(x)}\right)^2
    \quad\text{and}\quad
    \left(P(x) W(x)\right)'=
    2\left(L(x)-\frac{1}{2}Q'(x)\right)\left(\frac{R(x)}{Q(x)}\right)^2\;,
  \end{displaymath}
  we can rewrite~(\ref{eq:mom_con1}) as
  \begin{align*}
    &2\lambda_n\int_Ix^kQ(x)\left(W(x)\right)^2\left(Y_n(x)\right)^2\dd x
    +\frac{2\lambda_n}{k+1}\int_Ix^{k+1}
    \left(L(x)-\frac{1}{2}Q'(x)\right)
    \left(W(x)\right)^2\left(Y_n(x)\right)^2\dd x\\
    &\quad=O(M_n)\;.
  \end{align*}
  Since $\tau$ is chosen such that $\sqrt{\lambda}_n=O(n^{\tau})$
  as $n\to\infty$, it follows that
  \begin{align*}
    &\int_I
    \frac{x^kQ(x)\left(W(x)\right)^2\left(Y_n(x)\right)^2}
    {M_n\lambda_n}\dd x
    +\frac{1}{k+1}\int_I\frac{x^{k+1}
    \left(L(x)-\frac{1}{2}Q'(x)\right)
    \left(W(x)\right)^2\left(Y_n(x)\right)^2}{M_n\lambda_n}\dd x\\
    &\quad=O\left(\frac{1}{n^{4\tau}}\right)\;.
  \end{align*}
  Exploiting the property that $Q$ is a polynomial of degree at most
  two and that
  $L-\frac{1}{2}Q'$ is a polynomial of degree at most one to obtain expressions
  in terms of the limit moments and using the
  observation that the error term does not contribute as a result of
  \begin{displaymath}
    \lim_{N\to\infty}\sqrt{N}\sum_{n=N+1}^\infty\frac{1}{n^2}=0
    \quad\text{as well as}\quad
    \lim_{N\to\infty}N \sum_{n=N+1}^\infty\frac{1}{n^4}=0\;,
  \end{displaymath}
  we deduce
  the claimed recurrence relation for the limit moments
  $\{m_k:k\in\N_0\}$.
\end{proof}

According to the following proposition, under suitable boundary conditions,
weighted moments
of the function
$x\mapsto 1/\sqrt{P(x)W(x)}$ defined on the interval $I$ satisfy the
same recurrence relation as the limit moments considered above.
\begin{propn}\label{propn:conv_mom}
  Provided that, for all $l\in\N_0$ and for $c\in\{a,b\}$,
  \begin{equation}\label{eq:bound_cond_RQ}
    \lim_{x\to b}\frac{x^{l+1}R(x)}{\sqrt{Q(x)}}
    -\lim_{x\to a}\frac{x^{l+1}R(x)}{\sqrt{Q(x)}}=0\;,
  \end{equation}
  the weighted moments defined by, for $k\in\N_0$ and some constant $C\in\R$,
  \begin{displaymath}
    \widetilde{m}_k=C
    \int_I\frac{x^k\left(W(x)\right)^2}{\sqrt{P(x)W(x)}}\dd x
  \end{displaymath}
  satisfy the recurrence relation
  \begin{displaymath}
    (k+1)Q(0)\widetilde{m}_{k}+
    \left(L(0)+\left(k+\frac{1}{2}\right)Q'(0)\right)
    \widetilde{m}_{k+1}+
    \left(L'(0)+\frac{k}{2}Q''(0)\right)
    \widetilde{m}_{k+2}=0\;.
  \end{displaymath}
\end{propn}
\begin{proof}
  Since
  \begin{displaymath}
    \frac{\left(W(x)\right)^2}{\sqrt{P(x)W(x)}}
    =\frac{R(x)}{\left(Q(x)\right)^{\frac{3}{2}}}
    \quad\text{as well as}\quad
    \left(\frac{R(x)}{\sqrt{Q(x)}}\right)'
    =\frac{R(x)}{\left(Q(x)\right)^{\frac{3}{2}}}
    \left(L(x)-\frac{1}{2}Q'(x)\right)\;,
  \end{displaymath}
  integration by parts and the boundary
  condition~(\ref{eq:bound_cond_RQ}) give that, for all
  $k\in\N_0$,
  \begin{align*}
    \int_I\frac{x^kQ(x)\left(W(x)\right)^2}{\sqrt{P(x)W(x)}}\dd x
    =\int_I\frac{x^kR(x)}{\sqrt{Q(x)}}\dd x
    &=-\frac{1}{k+1}\int_I\frac{x^{k+1} R(x)}{\left(Q(x)\right)^{\frac{3}{2}}}
      \left(L(x)-\frac{1}{2}Q'(x)\right)\dd x\\
    &=-\frac{1}{k+1}\int_I
      \frac{x^{k+1}\left(L(x)-\frac{1}{2}Q'(x)\right)\left(W(x)\right)^2}
      {\sqrt{P(x)W(x)}}\dd x\;.
  \end{align*}
  As in the proof of Proposition~\ref{propn:moments}, this yields the
  claimed recurrence relation.
\end{proof}

The second main ingredient for our subsequent analysis
which we establish for general
classical orthogonal polynomial systems is the Christoffel--Darboux type
formula given in Proposition~\ref{propn:CDT}. For its derivation, we
need an extension of
Szeg\H{o}~\cite[Theorem~3.2.1]{szego} to orthogonal
polynomials which are not assumed to be orthonormal. It characterises
two of the three coefficients appearing in the three-term recurrence
relation satisfied by the orthogonal polynomials $\{Y_n: n\in\N_0\}$
in terms of the square $M_n$ of the $L^2(I,W(x)\dd x)$ norm
of $Y_n$ and the leading
coefficient $K_n$ of $Y_n$.
\begin{lemma}\label{lem:3TR}
  Let $\{Y_n: n\in\N_0\}$ be a family of classical orthogonal
  polynomials on the interval $I$ with
  respect to the weight function $W\colon I\to\R$.
  We have the three-term recurrence relation, for
  $n\in\N$ and $x\in I$,
  \begin{displaymath}
    Y_{n+1}(x)=(A_nx+B_n)Y_n(x)-C_nY_{n-1}(x)
  \end{displaymath}
  with constants $A_n,B_n,C_n\in\R$, where
  \begin{displaymath}
    A_n=\frac{K_{n+1}}{K_n}\quad\text{and}\quad
    C_n=\frac{K_{n-1}K_{n+1}}{K_n^2}\frac{M_n}{M_{n-1}}\;.
  \end{displaymath}
\end{lemma}
\begin{proof}
  Adapting Szeg\H{o}~\cite[Proof of Theorem~3.2.1]{szego}, we
  observe that $Y_{n+1}(x)-A_nxY_n(x)$ defines a polynomial on $I$ of
  degree $n$. Hence, this polynomial can be written as a linear
  combination of $Y_0, Y_1,\dots, Y_n$. Similarly, for all $m\in\N_0$,
  the polynomial defined by $xY_m(x)$ can be expressed as a linear
  combination of $Y_0, Y_1,\dots, Y_{m+1}$.
  By orthogonality, it follows that, as long as
  $0\leq m <n-1$,
  \begin{displaymath}
    \int_I \left(Y_{n+1}(x)-A_nxY_n(x)\right)Y_m(x)W(x)\dd x=0\;.
  \end{displaymath}
  Thus, the polynomial defined by $Y_{n+1}(x)-A_nxY_n(x)$ is a linear
  combination of $Y_{n-1}$ and $Y_n$ only. It remains to identify the
  coefficient in front of $Y_{n-1}$. Using the
  three-term recurrence relation and orthogonality, we deduce that
  \begin{equation}\label{eq:cond4CN}
    A_n\int_IxY_n(x)Y_{n-1}(x)W(x)\dd x-C_nM_{n-1}
    =\int_IY_{n+1}(x)Y_{n-1}(x)W(x)\dd x=0\;.
  \end{equation}
  When expressing the polynomial defined by $xY_{n-1}(x)$ as a
  linear combination of $Y_0, Y_1,\dots, Y_n$, the coefficient in
  front of $Y_n$ is $K_{n-1}/K_n$. We conclude
  \begin{displaymath}
    \int_IxY_n(x)Y_{n-1}(x)W(x)\dd x=\frac{K_{n-1}M_n}{K_n}\;,
  \end{displaymath}
  which together with~(\ref{eq:cond4CN}) yields
  \begin{displaymath}
    C_n=\frac{K_{n-1}K_{n+1}}{K_n^2}\frac{M_n}{M_{n-1}}\;,
  \end{displaymath}
  as required.
\end{proof}
The standard Christoffel--Darboux formula, see
Szeg\H{o}~\cite[Theorem~3.2.2]{szego}, asserts that, for $N\in\N$ and
$x,y\in I$,
\begin{displaymath}
  (x-y)\sum_{n=0}^N\frac{Y_n(x)Y_n(y)}{M_n}
  =\frac{K_N}{K_{N+1}M_N}\left(Y_{N+1}(x)Y_N(y)-Y_N(x)Y_{N+1}(y)\right)\;.
\end{displaymath}
The Christoffel--Darboux type formula stated below
extends the Christoffel--Darboux type
formula, see~\cite[Proposition~5.1]{semicircle}, derived for integrals
of Legendre polynomials. It
enters the analysis in exactly the same way
as~\cite[Proposition~5.1]{semicircle} did in
establishing the convergence away from the diagonal
for~\cite[Theorem~1.5]{semicircle}.
\begin{propn}\label{propn:CDT}
  Let $\{Y_n: n\in\N_0\}$ be a family of classical orthogonal
  polynomials and let $\{\lambda_n: n\in\N_0\}$ be the corresponding
  family of eigenvalues. Setting, for $n\in\N_0$ and $x,y\in I$,
  \begin{displaymath}
    D_{n+1}(x,y)=Y_{n+1}(x)Y_n(y)-Y_n(x)Y_{n+1}(y)\;,
  \end{displaymath}
  we have, for $N\in\N$,
  \begin{align*}
    &(x-y)\sum_{n=1}^N\frac{Y_n(x)Y_n(y)}{M_n\lambda_n}\\
    &\quad =\frac{K_N}{K_{N+1}M_N}\frac{D_{N+1}(x,y)}{\lambda_N}
    -\frac{K_0}{K_1M_0}\frac{D_1(x,y)}{\lambda_1}
    +\sum_{n=1}^{N-1}\frac{K_{n}}{K_{n+1}M_{n}}D_{n+1}(x,y)
      \left(\frac{1}{\lambda_{n}}-\frac{1}{\lambda_{n+1}}\right)\;.
  \end{align*}
\end{propn}
\begin{proof}
  Since the eigenvalues are known to be mutually distinct and as
  $\lambda_0=0$, the expressions we consider are all well-defined.
  From the three-term recurrence relation given in
  Lemma~\ref{lem:3TR} for the family $\{Y_n: n\in\N_0\}$ of classical
  orthogonal polynomials, it follows that, for $n\in\N$ and for $x,y\in I$,
  \begin{displaymath}
    D_{n+1}(x,y) =
    A_n(x-y)Y_n(x)Y_n(y)+C_n D_n(x,y)\;,
  \end{displaymath}
  which implies that, for $N\in\N$,
  \begin{displaymath}
    (x-y) \sum_{n=1}^N\frac{Y_n(x)Y_n(y)}{M_n\lambda_n}
    =\sum_{n=1}^N\frac{D_{n+1}(x,y)}{A_n M_n\lambda_n} -
    \sum_{n=1}^N\frac{C_nD_n(x,y)}{A_n M_n\lambda_n}\;.
  \end{displaymath}
  Using the expressions for $A_n$ and $C_n$ from Lemma~\ref{lem:3TR},
  we deduce
  \begin{align*}
    &(x-y) \sum_{n=1}^N\frac{Y_n(x)Y_n(y)}{M_n\lambda_n}\\
    &\quad=\sum_{n=1}^{N}\frac{K_{n}}{K_{n+1}M_{n}}
      \frac{D_{n+1}(x,y)}{\lambda_{n}} -
      \sum_{n=0}^{N-1}\frac{K_{n}}{K_{n+1}M_{n}}
      \frac{D_{n+1}(x,y)}{\lambda_{n+1}}\\
    &\quad=\frac{K_N}{K_{N+1}M_N}\frac{D_{N+1}(x,y)}{\lambda_N}
    -\frac{K_0}{K_1M_0}\frac{D_1(x,y)}{\lambda_1}
      +\sum_{n=1}^{N-1}\frac{K_{n}}{K_{n+1}M_{n}}D_{n+1}(x,y)
      \left(\frac{1}{\lambda_{n}}-\frac{1}{\lambda_{n+1}}\right)\;,
  \end{align*}
  as claimed.
\end{proof}
In particular, note that if, for $x,y\in I$, 
\begin{displaymath}
  \lim_{n\to\infty}\frac{K_n}{K_{n+1}M_n}\frac{D_{n+1}(x,y)}{\lambda_n}
  =0\;,
\end{displaymath}
the Christoffel--Darboux type formula yields
\begin{displaymath}
  (x-y)\sum_{n=N+1}^\infty\frac{Y_n(x)Y_n(y)}{M_n\lambda_n}
  =\sum_{n=N}^\infty\frac{K_{n}}{K_{n+1}M_{n}}D_{n+1}(x,y)
  \left(\frac{1}{\lambda_{n}}-\frac{1}{\lambda_{n+1}}\right)
  -\frac{K_N}{K_{N+1}M_N}\frac{D_{N+1}(x,y)}{\lambda_N}\;.
\end{displaymath}
This implication is so powerful in our subsequent
analysis because, as $n\to\infty$,
\begin{displaymath}
  \frac{1}{\lambda_{n}}-\frac{1}{\lambda_{n+1}}
  =O\left(\frac{1}{n^{2\tau+1}}\right)
\end{displaymath}
while $\lambda_n^{-1}=O\left(n^{-2\tau}\right)$, which gives rise to
better estimates for the terms in the second series above than for the
terms in the first series.
Essentially, the
Christoffel--Darboux type formula encompasses a cancellation which is
otherwise missed when bounding summands separately.

Having established two main results for general classical orthogonal
polynomial systems, we now consider the Hermite polynomials, the
associated Laguerre polynomials and the Jacobi polynomials separately,
in that order. The reason for choosing this order is that the proofs
for the Hermite polynomials are the cleanest and already convey the
used strategy well, whilst the proofs for the Jacobi polynomials
contain the most involved expressions.

Throughout the study in the following three subsections, we frequently
use Stirling's formula in the form, for $z\in\R$ as
$z\to\infty$,
\begin{equation}\label{eq:stirling}
  \Gamma(z+1)\sim \sqrt{2\pi z}\left(\frac{z}{\e}\right)^z
\end{equation}
as well as the asymptotics for the Gamma function that, for
$\alpha\in\R$ and as $z\to\infty$,
\begin{equation}\label{eq:Gamma_asymp}
  \Gamma(z+\alpha)\sim \Gamma(z)z^\alpha\;,
\end{equation}
see~\cite[6.1.39]{handbook}. We further need the Legendre duplication
formula~\cite[6.1.18]{handbook} asserting, for $z\in\R$,
\begin{equation}\label{eq:LDF}
  \Gamma(z)\Gamma\left(z+\frac{1}{2}\right)
  =2^{1-2z}\sqp\Gamma(2z)\;.
\end{equation}

\subsection{Hermite polynomials}
\label{sec:hermite}

The Hermite polynomials $\{H_n:n\in\N_0\}$ are the classical orthogonal
polynomials on the interval $I=\R$ with respect to the weight function
$W\colon \R\to\R$ given by
\begin{displaymath}
  W(x)=\exp\left(-x^2\right)
\end{displaymath}
subject to the normalisations, for $n\in\N_0$,
\begin{displaymath}
  M_n=\int_{-\infty}^\infty \e^{-x^2}\left(H_n(x)\right)^2\dd x
  =2^nn!\sqrt{\pi}
  \quad\text{and}\quad
  K_n=2^n\;,
\end{displaymath}
see Szeg\H{o}~\cite[Chapter~5.5]{szego}. These Hermite polynomials are
also referred to as the physicist Hermite polynomials to distinguish
them from the probabilist Hermite polynomials which are orthogonal on
$\R$ with respect to the weight function defined by
$\exp\left(-\frac{1}{2}x^2\right)$.

As discussed by Littlejohn and Krall in~\cite{littlekrall}, the
Hermite polynomials solve the differential equation, for $n\in\N_0$
and $x\in\R$,
\begin{displaymath}
  H_n''(x)-2xH_n'(x)+2nH_n(x)=0\;,
\end{displaymath}
whose Sturm--Liouville form is
\begin{equation}\label{eq:H_SL}
  \left(\e^{-x^2} H_n'(x)\right)'=-2n \e^{-x^2} H_n(x)\;.
\end{equation}
In particular, using the notations introduced previously, we have
\begin{displaymath}
  Q(x)=1\;,\quad
  L(x)=-2x\;,\quad
  P(x)=W(x)=\exp\left(-x^2\right)
  \quad\text{and}\quad
  \lambda_n=2n\;.
\end{displaymath}
For our analysis, we need the two identities, for $n\in\N$,
\begin{equation}\label{eq:H_diff}
  H_n'(x)=2nH_{n-1}(x)
\end{equation}
and
\begin{equation}\label{eq:H_recdiff}
  H_n(x)=2xH_{n-1}(x)-H_{n-1}'(x)\;,
\end{equation}
see~\cite[5.5.10]{szego}, giving rise to the
three-term recurrence relation
\begin{equation}\label{eq:H_rec}
  H_{n+1}(x)=2xH_n(x)-2nH_{n-1}(x)\;,
\end{equation}
which is consistent with Lemma~\ref{lem:3TR}.
We further rely on the asymptotic formula that,
for $x\in\R$ and as $n\to\infty$,
\begin{equation}\label{eq:H_asymp}
  H_n(x)=\e^{\frac{1}{2}x^2}\frac{2^n}{\sqrt{\pi}}
  \Gamma\left(\frac{n+1}{2}\right)
  \cos\left(x\sqrt{2n}-\frac{n\pi}{2}\right)
  \left(1+O\left(\frac{1}{\sqrt{n}}\right)\right)\;,
\end{equation}
a consequence of~\cite[13.5.16 and 13.6.38]{handbook}, where the
bound on the error term is uniform on every finite real interval.
By the Legendre
duplication formula~(\ref{eq:LDF}), we have
\begin{displaymath}
  \Gamma\left(\frac{n+1}{2}\right)
  \Gamma\left(\frac{n}{2}+1\right)
  =2^{-n}\sqp\Gamma(n+1)\;,
\end{displaymath}
which together with Stirling's formula~(\ref{eq:stirling}) implies that, as
$n\to\infty$, 
\begin{displaymath}
  \frac{2^n}{\sqrt{\pi}}
  \Gamma\left(\frac{n+1}{2}\right)
  \sim \sqrt{2}\left(\frac{2n}{\e}\right)^{\frac{n}{2}}\;.
\end{displaymath}
Thus, for all $K\in\R$ with $K>0$, there exists a positive constant $C\in\R$
such that, for all $n\in\N_0$ and for all $x\in[-K,K]$,
\begin{equation}\label{eq:H_bound}
  \left|H_n(x)\right|
  \leq C \left(\frac{2n}{\e}\right)^{\frac{n}{2}}\;.
\end{equation}
Moreover, we observe that due to the exponential decay in
$\exp\left(-x^2\right)$ as $x\to\pm\infty$, for all $l,n\in\N_0$,
\begin{equation}\label{eq:H_bc}
  \lim_{x\to\pm\infty}x^l\e^{-x^2}H_n(x)
  =\lim_{x\to\pm\infty}x^l\e^{-x^2}H_n'(x)=0\;.
\end{equation}

The first lemma in this subsection is
later needed to verify that the limit moments on the diagonal
satisfy the correct initial condition.
\begin{lemma}\label{lem:H_initial}
  For all $n\in\N_0$, we have
  \begin{displaymath}
    \int_{-\infty}^\infty\e^{-2x^2}\left(H_n(x)\right)^2\dd x
    =2^{n-\frac{1}{2}}\,\Gamma\left(n+\frac{1}{2}\right)\;.
  \end{displaymath}
\end{lemma}
\begin{proof}
  The Sturm--Liouville differential equation~(\ref{eq:H_SL}) yields,
  for $n\in\N_0$ and $x\in\R$,
  \begin{displaymath}
    \left(\e^{-2x^2}\left(H_n'(x)\right)^2\right)'
    =-4n\e^{-2x^2}H_n(x)H_n'(x)\;.
  \end{displaymath}
  Integrating by parts as well as using~(\ref{eq:H_bc}), the
  identity~(\ref{eq:H_diff}) and the recurrence
  relation~(\ref{eq:H_rec}), we obtain, for $n\in\N$,
  \begin{align*}
    \int_{-\infty}^\infty\e^{-2x^2}\left(H_n'(x)\right)^2\dd x
    &=2n\int_{-\infty}^\infty \e^{-2x^2}2x H_n(x)H_n'(x)\dd x\\
    &=4n^2\int_{-\infty}^\infty \e^{-2x^2}
      \left(H_{n+1}(x)+2nH_{n-1}(x)\right)H_{n-1}(x)\dd x\;.
  \end{align*}
  Applying~(\ref{eq:H_diff}) to the left hand side above
  implies that
  \begin{equation}\label{eq:H_help1}
    (1-2n) \int_{-\infty}^\infty\e^{-2x^2}\left(H_{n-1}(x)\right)^2\dd x
    =\int_{-\infty}^\infty\e^{-2x^2}H_{n+1}(x)H_{n-1}(x)\dd x\;.
  \end{equation}
  On the other hand, by using~(\ref{eq:H_diff}), integrating by
  parts and noting~(\ref{eq:H_bc}), and concluding with
  (\ref{eq:H_diff}) as well as (\ref{eq:H_recdiff}), we can argue, for
  $n\in\N$,
  \begin{align*}
    2(n+1)\int_{-\infty}^\infty\e^{-2x^2}\left(H_n(x)\right)^2\dd x
    &=\int_{-\infty}^\infty\e^{-2x^2}H_n(x)H_{n+1}'(x)\dd x\\
    &=\int_{-\infty}^\infty\e^{-2x^2}\left(
      4xH_n(x)-H_n'(x)\right)H_{n+1}(x)\dd x\\
    &=\int_{-\infty}^\infty\e^{-2x^2}\left(
      2H_{n+1}(x)+2n H_{n-1}(x)\right)H_{n+1}(x)\dd x\;.
  \end{align*}
  Putting this together with~(\ref{eq:H_help1}) gives, for $n\in\N$,
  \begin{align*}
    &\int_{-\infty}^\infty\e^{-2x^2}\left(H_{n+1}(x)\right)^2\dd x\\
    &\qquad=(n+1)\int_{-\infty}^\infty\e^{-2x^2}\left(H_n(x)\right)^2\dd x
    +n(2n-1)
    \int_{-\infty}^\infty\e^{-2x^2}\left(H_{n-1}(x)\right)^2\dd x\;,
  \end{align*}
  which is a recurrence relation for the integrals we wish to
  evaluate. Since
  \begin{align*}
    \int_{-\infty}^\infty\e^{-2x^2}\left(H_0(x)\right)^2\dd x
    &=\int_{-\infty}^\infty\e^{-2x^2}\dd x=\sqrt{\frac{\pi}{2}}
    =\frac{\Gamma\left(\frac{1}{2}\right)}{\sqrt{2}}\;,\\
    \int_{-\infty}^\infty\e^{-2x^2}\left(H_1(x)\right)^2\dd x
    &=\int_{-\infty}^\infty\e^{-2x^2}4x^2\dd x=\sqrt{\frac{\pi}{2}}
    =\sqrt{2}\,\Gamma\left(\frac{3}{2}\right)\;,
  \end{align*}
  and, for $n\in\N$,
  \begin{displaymath}
    2(n+1)\,\Gamma\left(n+\frac{1}{2}\right)
    +n(2n-1)\,\Gamma\left(n-\frac{1}{2}\right)
    =2\left(2n+1\right)\Gamma\left(n+\frac{1}{2}\right)
    =4\,\Gamma\left(n+\frac{3}{2}\right)\;,
  \end{displaymath}
  the claimed result follows by induction.
\end{proof}
The next two lemmas feed into the Arzel\`{a}--Ascoli theorem to deduce
a locally uniform convergence on the diagonal, which allows us to
establish the existence of a pointwise limit and to identify it
from the moment analysis.
\begin{lemma}\label{lem:H_unifbound1}
  Fix $K\in\R$ with $K>0$. The family
  \begin{displaymath}
    \left\{\sqrt{N}\sum_{n=N+1}^\infty
      \frac{H_n(x)H_n(y)}{2n\, 2^nn!}:
      N\in\N\text{ and }x,y\in [-K,K]\right\}
  \end{displaymath}
  is uniformly bounded.
\end{lemma}
\begin{proof}
  According to the bound~(\ref{eq:H_bound}), there exists a
  constant $C\in\R$
  such that, for all $n\in\N$ and for all $x,y\in [-K,K]$,
  \begin{displaymath}
    \left|\frac{H_n(x)H_n(y)}{2n\, 2^nn!}\right|
    \leq \frac{C^2}{2n\, 2^nn!}\left(\frac{2n}{\e}\right)^n\;.
  \end{displaymath}
  By Stirling's formula~(\ref{eq:stirling}), we have, as $n\to\infty$,
  \begin{displaymath}
    \frac{1}{2^nn!}\left(\frac{2n}{\e}\right)^n
    =\frac{1}{n!}\left(\frac{n}{\e}\right)^n
    \sim\frac{1}{\sqrt{2\pi n}}\;.
  \end{displaymath}
  Hence, there exists a positive constant $D\in\R$ such that, for all
  $n\in\N$ and for all $x,y\in [-K,K]$,
  \begin{equation}\label{eq:H_interbound}
    \left|\frac{H_n(x)H_n(y)}{2n\, 2^nn!}\right|
    \leq D n^{-\frac{3}{2}}\;,
  \end{equation}
  which implies that, for all $N\in\N$,
  \begin{displaymath}
    \left|\sqrt{N}\sum_{n=N+1}^\infty
      \frac{H_n(x)H_n(y)}{2n\, 2^nn!}\right|
    \leq \sqrt{N}\sum_{n=N+1}^\infty D n^{-\frac{3}{2}}
    \leq D\sqrt{N}\left(N^{-\frac{3}{2}}+\int_N^\infty
      z^{-\frac{3}{2}}\dd z\right)
    \leq 3D\;,
  \end{displaymath}
  as required.
\end{proof}
\begin{lemma}\label{lem:H_unifbound2}
  Fix $K\in\R$ with $K>0$. The family
  \begin{displaymath}
    \left\{\sqrt{N}\sum_{n=N+1}^\infty
      \frac{H_n(x)H_{n-1}(x)}{2^nn!}:
      N\in\N\text{ and }x\in[-K,K]\right\}
  \end{displaymath}
  is uniformly bounded.
\end{lemma}
\begin{proof}
  Using the recurrence relation~(\ref{eq:H_rec}), we can rewrite
  \begin{displaymath}
    \sum_{n=N+1}^\infty
    \frac{H_n(x)H_{n-1}(x)}{2^nn!}
    =\sum_{n=N+1}^\infty
    \frac{2x\left(H_n(x)\right)^2 - H_n(x)H_{n+1}(x)}{2n\, 2^nn!}\;,
  \end{displaymath}
  which implies that
  \begin{equation}\label{eq:H_help2}
    \sum_{n=N+1}^\infty
    \frac{H_n(x)H_{n-1}(x)}{2^n(n-1)!}
    \left(\frac{1}{n}+\frac{1}{n-1}\right)
    =\sum_{n=N+1}^\infty
    \frac{2x\left(H_n(x)\right)^2}{2n\, 2^nn!}
    +\frac{H_N(x)H_{N+1}(x)}{2N\, 2^{N}N!}\;.
  \end{equation}
  Since, for $n\in\N$ with $n\geq 2$,
  \begin{displaymath}
    \frac{1}{n}+\frac{1}{n-1}=\frac{1}{n}\left(2+\frac{1}{n-1}\right)\;,
  \end{displaymath}
  we deduce that
  \begin{displaymath}
    2\sum_{n=N+1}^\infty
    \frac{H_n(x)H_{n-1}(x)}{2^nn!}
    =\sum_{n=N+1}^\infty
    \frac{2x\left(H_n(x)\right)^2}{2n\, 2^nn!}
    -\sum_{n=N+1}^\infty\frac{H_n(x)H_{n-1}(x)}{(n-1)\,2^nn!}
    +\frac{H_N(x)H_{N+1}(x)}{2N\, 2^{N}N!}\;.
  \end{displaymath}
  Crucially, thanks to the telescoping-like rearrangement
  in~(\ref{eq:H_help2}), the terms of the third series in the above
  identity have an additional factor of $n-1$ in the denominator
  compared to the terms of the series we are interested in. As a
  result, we can use the bound~(\ref{eq:H_bound}) to
  conclude. For all $n\in\N$ with $n\geq 2$ and for all
  $x\in[-K,K]$, we have
  \begin{displaymath}
    \left|\frac{H_n(x)H_{n-1}(x)}{(n-1)\,2^nn!}\right|
    \leq \frac{C^2}{(n-1)\,2^nn!}
    \left(\frac{2n}{\e}\right)^{\frac{n}{2}}
    \left(\frac{2(n-1)}{\e}\right)^{\frac{n-1}{2}}\;.
  \end{displaymath}
  From Stirling's formula~(\ref{eq:stirling}), it follows that there
  exists $D\in\R$ such that, for all $n\in\N$ with $n\geq 2$ and for all
  $x\in[-K,K]$,
  \begin{displaymath}
    \left|\frac{H_n(x)H_{n-1}(x)}{(n-1)\,2^nn!}\right|
    \leq Dn^{-2}\;.
  \end{displaymath}
  Moreover, we can choose $D\in\R$ such that, for all
  $N\in\N$ and for all $x\in[-K,K]$, 
  \begin{displaymath}
    \left|\frac{H_N(x)H_{N+1}(x)}{2N\, 2^{N}N!}\right|
    \leq D N^{-1}\;.
  \end{displaymath}
  This together with the bound~(\ref{eq:H_interbound}) obtained in the
  proof of Lemma~\ref{lem:H_unifbound1} implies that
  \begin{displaymath}
    \left|\sqrt{N}
      \sum_{n=N+1}^\infty\frac{H_n(x)H_{n-1}(x)}{2^nn!}
    \right|
    \leq \sqrt{N}\left(
      \sum_{n=N+1}^\infty DK n^{-\frac{3}{2}}+
      \sum_{n=N+1}^\infty Dn^{-2}+
      D N^{-1}
    \right)\leq 3D(1+K)\;,
  \end{displaymath}
  which yields the claimed result.
\end{proof}

Putting everything together, we can determine the asymptotic error in
approximating the Green's function associated with the
Hermite polynomials by truncating the bilinear expansion.
\begin{proof}[Proof of Theorem~\ref{thm:hermite}]
  We start by proving the claimed convergence on the diagonal.
  For $N\in\N_0$, let $S_N\colon \R\to\R$ be defined by
  \begin{displaymath}
    S_N(x)=\sqrt{N}\sum_{n=N+1}^\infty
    \frac{\left(H_n(x)\right)^2}{2n\,2^nn!\sqrt{\pi}}\;.
  \end{displaymath}
  We have, for all $x\in\R$, 
  \begin{displaymath}
    S_N'(x)=\sqrt{N}\sum_{n=N+1}^\infty
    \frac{H_n(x)H_{n-1}(x)}{2^{n-1}n!\sqrt{\pi}}
  \end{displaymath}
  due to~(\ref{eq:H_diff}) and since the series converges uniformly on
  finite intervals by the estimates given in the proof of
  Lemma~\ref{lem:H_unifbound2}.
  Given Lemma~\ref{lem:H_unifbound1},
  Lemma~\ref{lem:H_unifbound2} and the Arzel\`{a}--Ascoli theorem,
  it remains to establish the convergence of moments to conclude that
  the functions $S_N$ converge pointwise on $\R$
  as $N\to\infty$ to the claimed limit function.
  The symmetry
  of the Hermite polynomials, implied by the
  analogue~\cite[5.5.3]{szego} of the Rodrigues
  formula, shows that, for all $k\in\N_0$ and
  for all $N\in\N_0$,
  \begin{displaymath}
    \int_{-\infty}^\infty x^{2k+1}\e^{-2x^2}S_N(x)\dd x=0\;,
  \end{displaymath}
  which is consistent with the claimed limit function being
  even. Moreover, for the even weighted moments and in terms of the
  limit moments considered in Proposition~\ref{propn:moments}, we have
  \begin{displaymath}
    m_{2k}=\lim_{N\to\infty}
    \int_{-\infty}^\infty x^{2k}\e^{-2x^2}S_N(x)\dd x\;,
  \end{displaymath}
  where we can interchange summation and integration because all terms
  are non-negative. Since, for $l\in\N_0$ fixed, the function
  \begin{displaymath}
    x\mapsto \left(\left(x^l\e^{-x^2}\right)'\e^{-x^2}\right)'\e^{x^2}
  \end{displaymath}
  is bounded on $\R$ and using~(\ref{eq:H_bc}), we see that, as $n\to\infty$,
  \begin{displaymath}
    \int_{-\infty}^\infty\left(x^l
      \e^{-x^2}\right)'\e^{-x^2}H_n(x)H'_n(x)\dd x
    =-\frac{1}{2}
    \int_{-\infty}^\infty\left(\left(x^l\e^{-x^2}\right)'\e^{-x^2}\right)'
    \left(H_n(x)\right)^2\dd x=O\left(M_n\right)\;.
  \end{displaymath}
  This together with~(\ref{eq:H_bc}) shows that we can employ
  Proposition~\ref{propn:moments} to deduce that the even limit moments
  satisfy the recurrence relation
  \begin{displaymath}
    m_{2k+2}=\frac{k+1}{2}m_{2k}\;,
  \end{displaymath}
  which, according to Proposition~\ref{propn:conv_mom} and since
  $R(x)=\exp\left(-x^2\right)$ as well as $Q(x)=1$, is the same
  recurrence relation as satisfied by the even weighted moments
  \begin{displaymath}
    \widetilde{m}_{2k}=\frac{1}{\sqrt{2}\pi}
    \int_{-\infty}^\infty x^{2k}\e^{-x^2}\dd x\;.
  \end{displaymath}
  It suffices to prove $m_0=\widetilde{m}_{0}$ to conclude that
  the functions $S_N$ converge pointwise as $N\to\infty$ to the
  claimed limit function.
  Applying Lemma~\ref{lem:H_initial} and the
  asymptotics~(\ref{eq:Gamma_asymp}), we obtain
  \begin{displaymath}
    m_0=\lim_{N\to\infty}\sqrt{N}\sum_{n=N+1}^\infty
    \frac{2^{n-\frac{1}{2}}}{2n\,2^nn!\sqrt{\pi}}\,
    \Gamma\left(n+\frac{1}{2}\right)
    =\lim_{N\to\infty}\sqrt{N}\sum_{n=N+1}^\infty
    \frac{1}{2n\sqrt{2n\pi}}=\frac{1}{\sqrt{2\pi}}\;,
  \end{displaymath}
  and a direct computation gives
  \begin{displaymath}
    \widetilde{m}_0
    =\frac{1}{\sqrt{2}\pi}\int_{-\infty}^\infty\e^{-x^2}\dd x
    =\frac{1}{\sqrt{2\pi}}\;.
  \end{displaymath}
  It remains to establish the off-diagonal convergence, which is a
  consequence of Proposition~\ref{propn:CDT} and the
  bound~(\ref{eq:H_bound}). Due to the estimate~(\ref{eq:H_bound}) and
  Stirling's
  formula~(\ref{eq:stirling}), for $x,y\in\R$ fixed, there exists a
  positive constant $D\in\R$ such that, for all
  $n\in\N$,
  \begin{equation}\label{eq:H_diffbound}
    \left|
      \frac{H_{n+1}(x)H_n(y)-H_n(x)H_{n+1}(y)}
      {2n\,2^{n+1}n!}
    \right|\leq Dn^{-1}\;.
  \end{equation}
  Thus, applying Proposition~\ref{propn:CDT} to the Hermite
  polynomials and using
  \begin{displaymath}
    \frac{1}{\lambda_{n}}-\frac{1}{\lambda_{n+1}}
    =\frac{1}{2n}-\frac{1}{2(n+1)}
    =\frac{1}{2n(n+1)}\;,
  \end{displaymath}
  we have that, for all $N\in\N$,
  \begin{displaymath}
    (x-y)\sum_{n=N+1}^\infty
    \frac{H_n(x)H_n(y)}{2n\,2^nn!\sqrt{\pi}}
    =\sum_{n=N}^\infty\frac{1}{2n(n+1)}  
    \frac{D_{n+1}(x,y)}{2^{n+1}n!\sqrt{\pi}}
    -\frac{1}{2^{N+1}N!\sqrt{\pi}}\frac{D_{N+1}(x,y)}{2N}\;.
  \end{displaymath}
  It follows that
  \begin{displaymath}
    \left|(x-y)\sum_{n=N+1}^\infty
      \frac{H_n(x)H_n(y)}{2n\,2^nn!\sqrt{\pi}}\right|
    \leq D\left(\sum_{n=N}^\infty\frac{1}{n(n+1)}+\frac{1}{N}\right)
    =\frac{2D}{N}\;,
  \end{displaymath}
  which implies that provided $x\not= y$, for all $\gamma<1$,
  \begin{displaymath}
    \lim_{N\to\infty}N^\gamma\sum_{n=N+1}^\infty
    \frac{H_n(x)H_n(y)}{2n\,2^nn!\sqrt{\pi}}=0\;,
  \end{displaymath}
  as claimed.
\end{proof}
The proof of Theorem~\ref{thm:hermite} illustrates two aspects very
nicely. Firstly, it demonstrates why we weigh the moments considered
in Proposition~\ref{propn:moments} and
Proposition~\ref{propn:conv_mom} by $W^2$ instead of just by $W$ as
this is needed to ensure that the moments remain finite. Secondly, by
using the Christoffel--Darboux type formula given in
Proposition~\ref{propn:CDT} we improve the off-diagonal convergence in
Theorem~\ref{thm:hermite} by one factor of $N$ compared to what we
could deduce from the bound~(\ref{eq:H_bound}) alone.
\begin{remark}\rm
  By employing a similar workaround to the one we use when
  studying the Jacobi polynomials, it is possible to deduce the
  convergence on the diagonal in
  Theorem~\ref{thm:hermite} from the asymptotic
  formula~(\ref{eq:H_asymp}) and Lemma~\ref{lem:H_unifbound1}
  as well as Lemma~\ref{lem:H_unifbound2}. The approach
  presented above still has its own benefits as it nicely demonstrates
  Proposition~\ref{propn:moments}, it only uses~(\ref{eq:H_bound}) as
  opposed to the full asymptotic formula~(\ref{eq:H_asymp}), and it
  led us to Lemma~\ref{lem:H_initial}.
\end{remark}

\subsection{Associated Laguerre polynomials}
\label{sec:laguerre}

For a parameter $\alpha\in\R$ with $\alpha>-1$, the associated
Laguerre polynomials $\{L_n^{(\alpha)}:n\in\N_0\}$ are the classical
orthogonal polynomials on $I=[0,\infty)$ with respect to the weight
function $W\colon (0,\infty)\to\R$ defined by
\begin{displaymath}
  W(x)=x^{\alpha}\e^{-x}
\end{displaymath}
which are normalised such that, for $n\in\N_0$,
\begin{displaymath}
  M_n=\int_0^\infty x^{\alpha}\e^{-x}
  \left(L_n^{(\alpha)}(x)\right)^2\dd x=\frac{\Gamma(n+\alpha+1)}{n!}
  \quad\text{and}\quad
  K_n=\frac{(-1)^n}{n!}\;,
\end{displaymath}
see \cite[Chapter~5.1]{szego}. We adapt the convention that the
superscript for the associated Laguerre polynomials is reserved
for the parameter $\alpha>-1$, and not to indicate any number of
derivatives. According
to Littlejohn and Krall~\cite{littlekrall}, the associated Laguerre
polynomials
solve the differential equation, for $n\in\N_0$ and $x\in[0,\infty)$,
\begin{equation}\label{eq:AL_DE}
  x\,\frac{\db^2L_n^{(\alpha)}(x)}{\db x^2}+
  (\alpha+1-x)\,\frac{\db L_n^{(\alpha)}(x)}{\db x}+
  nL_n^{(\alpha)}(x)=0\;,
\end{equation}
which rewrites as
\begin{displaymath}
  \frac{\db}{\db x}
  \left(x^{\alpha+1}\e^{-x}\frac{\db L_n^{(\alpha)}(x)}{\db x}\right)
  =-n x^\alpha \e^{-x}L_n^{(\alpha)}(x)\;.
\end{displaymath}
In particular, we have
\begin{displaymath}
  Q(x)=x\;,\quad
  L(x)=\alpha+1-x\;,\quad
  P(x)=x^{\alpha+1}\e^{-x}
  \quad\text{and}\quad
  \lambda_n=n\;.
\end{displaymath}
Furthermore, the associated Laguerre polynomials satisfy the recurrence
relation, for $n\in\N$ and $x\in[0,\infty)$,
\begin{equation}\label{eq:AL_rec}
  (n+1)L_{n+1}^{(\alpha)}(x)
  =(2n+\alpha+1-x)L_n^{(\alpha)}(x)-
  (n+\alpha)L_{n-1}^{(\alpha)}(x)
\end{equation}
as well as the identity, for $n\in\N$ and $x\in[0,\infty)$,
\begin{equation}\label{eq:AL_diff}
  x\,\frac{\db L_n^{(\alpha)}(x)}{\db x}
  =n L_n^{(\alpha)}(x) - (n+\alpha)L_{n-1}^{(\alpha)}(x)\;,
\end{equation}
see~\cite[5.1.10 and 5.1.14]{szego}.
The bounds we use in our subsequent analysis for
the associated Laguerre polynomials
are consequences of the asymptotic
formula~\cite[Theorem~8.22.1]{szego}, which is due to
Fej\'er~\cite{fejer2,fejer1}.
\begin{thm}[Fej\'er's formula]\label{thm:fejer}
  For all $\alpha\in\R$ with $\alpha>-1$, we have, as $n\to\infty$,
  \begin{displaymath}
    L_n^{(\alpha)}(x)
    =\frac{\e^{\frac{x}{2}}n^{\frac{\alpha}{2}-\frac{1}{4}}}
    {\sqp x^{\frac{\alpha}{2}+\frac{1}{4}}}
    \cos\left(2\sqrt{n x}-\frac{\alpha\pi}{2}
      -\frac{\pi}{4}\right)
    +O\left(n^{\frac{\alpha}{2}-\frac{3}{4}}\right)\;,
  \end{displaymath}
  where the bound on the error term is uniform in $x\in[\eps,K]$ for
  $\eps, K\in\R$ with $0<\eps<K$.
\end{thm}
Due to the exponential decay in $\exp\left(-x\right)$ as $x\to\infty$, we
further obtain that, for all $l,n\in\N_0$,
\begin{equation}\label{eq:AL_bc1}
  \lim_{x\to\infty}x^lx^\alpha\e^{-x}L_n^{(\alpha)}(x)
  =\lim_{x\to\infty}x^lx^\alpha\e^{-x}\frac{\db L_n^{(\alpha)}(x)}{\db x}=0\;,
\end{equation}
and since $\alpha>-1$, we have, for all $l,n\in\N_0$,
\begin{equation}\label{eq:AL_bc2}
  \lim_{x\to 0}x^lx^{\alpha+1}\e^{-x}L_n^{(\alpha)}(x)
  =\lim_{x\to 0}x^lx^{\alpha+1}\e^{-x}\frac{\db L_n^{(\alpha)}(x)}{\db x}=0\;,
\end{equation}
which yield the boundary conditions needed for
Proposition~\ref{propn:moments} to deduce the recurrence relation
satisfied by the limit moments on the diagonal.
As for the Hermite polynomials, we compute the zeroth limit moment
explicitly, which makes use of the following two
results. The first lemma feeds directly into the second one,
but we include it
as a separate statement because the result is used a second time in
the proof of Theorem~\ref{thm:laguerre}.
\begin{lemma}\label{lem:AL_ind}
  For all $\alpha\in\R$ with $\alpha>-1$ and for all $n\in\N_0$, we have
  \begin{displaymath}
    \int_0^\infty x^{2\alpha+1}\e^{-2x}
    L_n^{(\alpha)}(x)\,\frac{\db L_n^{(\alpha)}(x)}{\db x}\dd x=0\;.
  \end{displaymath}
\end{lemma}
\begin{proof}
  Note that the integrals we consider are well-defined since $2\alpha+1>-1$.
  Integration by parts and~(\ref{eq:AL_bc1})
  yield, for $n\in\N_0$,
  \begin{displaymath}
    \int_0^\infty x^{2\alpha+1}\e^{-2x}
    \left(\frac{\db L_n^{(\alpha)}(x)}{\db x}\right)^2\dd x
    =-\int_0^\infty x\left(x^{2\alpha+1}\e^{-2x}
      \left(\frac{\db L_n^{(\alpha)}(x)}{\db x}\right)^2\right)'\dd x\;.
  \end{displaymath}
  Using the differential
  equation~(\ref{eq:AL_DE}) satisfied by the associated Laguerre
  polynomials, we deduce
  \begin{displaymath}
    \int_0^\infty x^{2\alpha+1}\e^{-2x}
    \left(\frac{\db L_n^{(\alpha)}(x)}{\db x}\right)^2\dd x
    =\int_0^\infty x^{2\alpha+1}\e^{-2x}\left(
      2nL_n^{(\alpha)}(x)+\frac{\db L_n^{(\alpha)}(x)}{\db x}\right)
    \frac{\db L_n^{(\alpha)}(x)}{\db x}\dd x\;,
  \end{displaymath}
  which implies the claimed result for $n\not= 0$. For the case $n=0$,
  it suffices to note that $L_0^{(\alpha)}(x)=1$ for all
  $x\in[0,\infty)$.
\end{proof}
Due to the identity~(\ref{eq:AL_diff}), Lemma~\ref{lem:AL_ind} shows
that, if $\alpha>-\frac{1}{2}$ then, for all $n\in\N$,
\begin{displaymath}
  n\int_0^\infty x^{2\alpha}\e^{-2x}
  \left(L_n^{(\alpha)}(x)\right)^2\dd x
  =(n+\alpha)\int_0^\infty x^{2\alpha}\e^{-2x}
  L_n^{(\alpha)}(x) L_{n-1}^{(\alpha)}(x)\dd x\;.
\end{displaymath}
While the integrals are not well-defined when
$\alpha\in(-1,-\frac{1}{2}]$, we still obtain,
for $\alpha>-1$ and $n\in\N$ fixed,
as $\eps\to 0$,
\begin{equation}\label{eq:AL_ind}
  n\int_\eps^\infty x^{2\alpha}\e^{-2x}
  \left(L_n^{(\alpha)}(x)\right)^2\dd x
  -(n+\alpha)\int_\eps^\infty x^{2\alpha}\e^{-2x}
  L_n^{(\alpha)}(x) L_{n-1}^{(\alpha)}(x)\dd x
  =O\left(\eps^{2\alpha+2}\right)\;.
\end{equation}
This is used in the proof of the next lemma, where extra care is
needed to ensure that it also works for $\alpha\in(-1,-\frac{1}{2}]$.
\begin{lemma}\label{lem:AL_initial}
  For $\alpha\in\R$ with $\alpha>-1$ and for all $n\in\N_0$, we have
  \begin{displaymath}
    \int_0^\infty x^{2\alpha+1}\e^{-2x}
    \left(L_{n}^{(\alpha)}(x)\right)^2\dd x
    =\frac{\Gamma(n+\alpha+1)\Gamma\left(n+\frac{1}{2}\right)
      \Gamma\left(\alpha+\frac{3}{2}\right)}
    {2\pi\left(\Gamma(n+1)\right)^2}\;.
  \end{displaymath}
\end{lemma}
\begin{proof}
  For $\eps>0$ and $n\in\N_0$, set
  \begin{displaymath}
    e_{n,\eps}^{(\alpha)}=-\left.\frac{1}{2}x^{2\alpha+1}\e^{-2x}
      \left(L_n^{(\alpha)}(x)\right)^2\right|_{x=\eps}\;.
  \end{displaymath}
  Integration by parts and~(\ref{eq:AL_bc1}) give
  \begin{align*}
    &\int_\eps^\infty x^{2\alpha+1}\e^{-2x}
    \left(L_n^{(\alpha)}(x)\right)^2\dd x
    =e_{n,\eps}^{(\alpha)}+\frac{1}{2}\int_\eps^\infty\e^{-2x}
      \left(x^{2\alpha+1}\left(L_n^{(\alpha)}(x)\right)^2\right)'\dd x\\
    &\quad=e_{n,\eps}^{(\alpha)}+\left(\alpha+\frac{1}{2}\right)
      \int_\eps^\infty x^{2\alpha}\e^{-2x}
      \left(L_n^{(\alpha)}(x)\right)^2\dd x+
      \int_\eps^\infty
      x^{2\alpha+1}e^{-2x}L_n^{(\alpha)}(x)
      \frac{\db L_n^{(\alpha)}(x)}{\db x}\dd x\;,
  \end{align*}
  which by Lemma~\ref{lem:AL_ind} implies that, as $\eps\to 0$,
  \begin{align}\label{eq:AL_help1}
    \begin{aligned}
      &\int_\eps^\infty x^{2\alpha+1}\e^{-2x}
      \left(L_n^{(\alpha)}(x)\right)^2\dd x\\
      &\qquad=e_{n,\eps}^{(\alpha)}+\left(\alpha+\frac{1}{2}\right)
      \int_\eps^\infty x^{2\alpha}\e^{-2x}
      \left(L_n^{(\alpha)}(x)\right)^2\dd x
      +O\left(\eps^{2\alpha+2}\right)\;.
    \end{aligned}
  \end{align}
  On the other hand, using the recurrence relation~(\ref{eq:AL_rec}),
  we obtain
  \begin{align*}
    \begin{aligned}
      &\int_\eps^\infty x^{2\alpha+1}\e^{-2x}
      \left(L_n^{(\alpha)}(x)\right)^2\dd x\\
      &\qquad =\int_\eps^\infty x^{2\alpha}\e^{-2x}\left(
        (2n+\alpha+1)L_n^{(\alpha)}(x)-
        (n+1)L_{n+1}^{(\alpha)}(x)-
        (n+\alpha)L_{n-1}^{(\alpha)}(x)
      \right)L_n^{(\alpha)}(x)\dd x\;.
    \end{aligned}
  \end{align*}
  From~(\ref{eq:AL_ind}), it follows that
  \begin{align}\label{eq:AL_help2}
    \begin{aligned}
      &\int_\eps^\infty x^{2\alpha+1}\e^{-2x}
      \left(L_n^{(\alpha)}(x)\right)^2\dd x\\
      &\qquad=\int_\eps^\infty x^{2\alpha}\e^{-2x}\left(
        (n+\alpha+1)L_n^{(\alpha)}(x)-
        (n+1)L_{n+1}^{(\alpha)}(x)\right)L_n^{(\alpha)}(x)\dd x
      +O\left(\eps^{2\alpha+2}\right)\;.
    \end{aligned}
  \end{align}
  Subtracting~(\ref{eq:AL_help1}) from (\ref{eq:AL_help2}), we
  conclude
  \begin{align*}
    &\left(n+\frac{1}{2}\right)\int_\eps^\infty x^{2\alpha}\e^{-2x}
    \left(L_n^{(\alpha)}(x)\right)^2\dd x\\
    &\qquad=e_{n,\eps}^{(\alpha)}+(n+1)\int_\eps^\infty x^{2\alpha}\e^{-2x}
    L_n^{(\alpha)}(x) L_{n+1}^{(\alpha)}(x)\dd x
      +O\left(\eps^{2\alpha+2}\right)\;,
  \end{align*}
  which together with~(\ref{eq:AL_ind}) and (\ref{eq:AL_help1})
  yields, as $\eps\to\infty$,
  \begin{align*}
    \begin{aligned}
      &(n+\alpha+1)\left(n+\frac{1}{2}\right)
      \int_\eps^\infty x^{2\alpha+1}\e^{-2x}
      \left(L_{n}^{(\alpha)}(x)\right)^2\dd x\\
      &\qquad =(n+1)^2\int_\eps^\infty x^{2\alpha+1}\e^{-2x}
      \left(L_{n+1}^{(\alpha)}(x)\right)^2\dd x
      +(n+\alpha+1)^2e_{n,\eps}^{(\alpha)}
      -(n+1)^2e_{n+1,\eps}^{(\alpha)}
      +O\left(\eps^{2\alpha+2}\right)\;.
    \end{aligned}
  \end{align*}
  According to~\cite[5.1.7]{szego}, the associated Laguerre
  polynomials satisfy
  \begin{displaymath}
    L_{n}^{(\alpha)}(0)=\binom{n+\alpha}{n}\;,
  \end{displaymath}
  which implies that, as $\eps\to 0$,
  \begin{displaymath}
    (n+\alpha+1)^2e_{n,\eps}^{(\alpha)}
    -(n+1)^2e_{n+1,\eps}^{(\alpha)}
    =O\left(\eps^{2\alpha+2}\right)\;.
  \end{displaymath}
  Thus, we can take the limit as $\eps\to\infty$ to deduce, for $n\in\N_0$,
  \begin{equation}\label{eq:AL_iteration}
    \int_0^\infty x^{2\alpha+1}\e^{-2x}
    \left(L_{n+1}^{(\alpha)}(x)\right)^2\dd x
    =\frac{\left(n+\alpha+1\right)\left(n+\frac{1}{2}\right)}
    {\left(n+1\right)^2}
    \int_0^\infty x^{2\alpha+1}\e^{-2x}
    \left(L_n^{(\alpha)}(x)\right)^2\dd x\;.
  \end{equation}
  We further compute that
  \begin{equation}\label{eq:AL_base}
    \int_0^\infty x^{2\alpha+1}\e^{-2x}
    \left(L_0^{(\alpha)}(x)\right)^2\dd x
    =\int_0^\infty x^{2\alpha+1}\e^{-2x}\dd x
    =\frac{\Gamma(2\alpha+2)}{2^{2\alpha+2}}\;.
  \end{equation}
  Using the Legendre duplication formula~(\ref{eq:LDF}) and
  $\Gamma\left(\frac{1}{2}\right)=\sqrt{\pi}$, this can be rewritten as
  \begin{displaymath}
    \frac{\Gamma(2\alpha+2)}{2^{2\alpha+2}}
    =\frac{\Gamma(\alpha+1)\Gamma\left(\frac{1}{2}\right)
      \Gamma\left(\alpha+\frac{3}{2}\right)}
    {2\pi}\;.
  \end{displaymath}
  The claimed identity then follows by induction from
  (\ref{eq:AL_iteration}) and (\ref{eq:AL_base}).
\end{proof}
Note that the above proof simplifies significantly for
$\alpha>-\frac{1}{2}$ because in that regime all the integrals
considered are
well-defined on
$[0,\infty)$ and the error terms $e_{n,\eps}^{(\alpha)}$ vanish as $\eps\to 0$.

For the Laguerre polynomials $\{L_n:n\in\N_0\}$, which are
the associated Laguerre polynomials with parameter $\alpha=0$,
Lemma~\ref{lem:AL_initial} as well as the identity (\ref{eq:AL_help1})
show that
\begin{displaymath}
  \int_0^\infty \e^{-2x}
  \left(L_{n}(x)\right)^2\dd x
  =\frac{1}{2^{2n+1}}\binom{2n}{n}
  \quad\text{and}\quad
  \int_0^\infty x\e^{-2x}
  \left(L_{n}(x)\right)^2\dd x
  =\frac{1}{4^{n+1}}\binom{2n}{n}\;.
\end{displaymath}

The next two results are the local uniform bounds later needed
to apply the Arzel\`{a}--Ascoli theorem.
\begin{lemma}\label{lem:AL_unifbound1}
  Fix $\eps,K\in\R$ with $0<\eps<K$ and $\alpha\in\R$ with
  $\alpha>-1$. The family
  \begin{displaymath}
    \left\{\sqrt{N}\sum_{n=N+1}^\infty
      \frac{\Gamma(n)L_n^{(\alpha)}(x)L_n^{(\alpha)}(y)}
    {\Gamma(n+\alpha+1)}:
      N\in\N\text{ and }x,y\in [\eps,K]\right\}
  \end{displaymath}
  is uniformly bounded.
\end{lemma}
\begin{proof}
  As a consequence of Fej\'er's formula, see
  Theorem~\ref{thm:fejer}, there exists a positive constant $C\in\R$,
  depending on $\eps$ and $K$, such
  that, for all $n\in\N$ and for all $x,y\in [\eps,K]$,
  \begin{displaymath}
    \left|L_n^{(\alpha)}(x)L_n^{(\alpha)}(y)\right|
    \leq Cn^{\alpha-\frac{1}{2}}\;.
  \end{displaymath}
  Using the asymptotics~(\ref{eq:Gamma_asymp}) for the Gamma function,
  we deduce that $C\in\R$ can further be chosen such that
  \begin{displaymath}
    \left|\frac{\Gamma(n)L_n^{(\alpha)}(x)L_n^{(\alpha)}(y)}
      {\Gamma(n+\alpha+1)}\right|
    \leq C n^{-\alpha-1}n^{\alpha-\frac{1}{2}}
    =Cn^{-\frac{3}{2}}\;.
  \end{displaymath}
  It follows that, for all $N\in\N$ and for all $x,y\in [\eps,K]$,
  \begin{displaymath}
    \left|\sqrt{N}\sum_{n=N+1}^\infty
      \frac{\Gamma(n)L_n^{(\alpha)}(x)L_n^{(\alpha)}(y)}
    {\Gamma(n+\alpha+1)}\right|
    \leq\sqrt{N}\sum_{n=N+1}^\infty C n^{-\frac{3}{2}}
    \leq 3C\;,
  \end{displaymath}
  which proves the claimed uniform boundedness.
\end{proof}
\begin{lemma}\label{lem:AL_unifbound2}
  Fix $\eps,K\in\R$ with $0<\eps<K$ and $\alpha\in\R$ with
  $\alpha>-1$. The family
  \begin{displaymath}
    \left\{\sqrt{N}\sum_{n=N+1}^\infty
      \frac{\Gamma(n)}{\Gamma(n+\alpha+1)}L_n^{(\alpha)}(x)
      \frac{\db L_n^{(\alpha)}(x)}{\db x}:
      N\in\N\text{ and }x\in [\eps,K]\right\}
  \end{displaymath}
  is uniformly bounded.
\end{lemma}
\begin{proof}
  According to~(\ref{eq:AL_diff}), we have, for all $n\in\N$,
  \begin{equation}\label{eq:AL_help3}
    \frac{\Gamma(n)}{\Gamma(n+\alpha+1)}L_n^{(\alpha)}(x)
    \frac{\db L_n^{(\alpha)}(x)}{\db x}
    =\frac{\Gamma(n+1)}{\Gamma(n+\alpha+1)}
    \frac{L_n^{(\alpha)}(x)\left(
        nL_n^{(\alpha)}(x)-(n+\alpha)L_{n-1}^{(\alpha)}(x)
      \right)}{nx}\;.
  \end{equation}
  We proceed by using a telescoping-like series to find an alternative
  expression
  for the series we are interested in such
  that Fej\'er's formula provides
  sufficient control over the terms of the new series to imply the
  claimed uniform boundedness.
  We observe that, for $N\in\N$,
  \begin{align*}
    &\sum_{n=N+1}^\infty
      \frac{\Gamma(n+1)}{\Gamma(n+\alpha+1)}
      \frac{(n+\alpha)L_n^{(\alpha)}(x)L_{n-1}^{(\alpha)}(x)}{n}
      +\sum_{n=N+1}^\infty
      \frac{\Gamma(n+1)}{\Gamma(n+\alpha+1)}
      \frac{(n+1)L_n^{(\alpha)}(x)L_{n+1}^{(\alpha)}(x)}{n}\\
    &\qquad=
      \sum_{n=N+1}^\infty
      \frac{\Gamma(n+1)L_n^{(\alpha)}(x)L_{n-1}^{(\alpha)}(x)}
      {\Gamma(n+\alpha)}
      \left(\frac{1}{n}+\frac{1}{n-1}\right)
      -\frac{\Gamma(N+2)}{\Gamma(N+\alpha+1)}
      \frac{L_N^{(\alpha)}(x)L_{N+1}^{(\alpha)}(x)}{N}\\
    &\qquad=
      \sum_{n=N+1}^\infty
      \frac{\Gamma(n)L_n^{(\alpha)}(x)L_{n-1}^{(\alpha)}(x)}
      {\Gamma(n+\alpha)}
      \left(2+\frac{1}{n-1}\right)
      -\frac{\Gamma(N+2)}{\Gamma(N+\alpha+1)}
      \frac{L_N^{(\alpha)}(x)L_{N+1}^{(\alpha)}(x)}{N}\;.  
  \end{align*}
  Applying the recurrence relation~(\ref{eq:AL_rec}), we deduce
  \begin{align}\label{eq:AL_help4}
    \begin{aligned}
      &\sum_{n=N+1}^\infty
        \frac{\Gamma(n+1)\left(L_n^{(\alpha)}(x)\right)^2}
        {\Gamma(n+\alpha+1)}\left(2+\frac{\alpha+1}{n}\right)
        -\sum_{n=N+1}^\infty
        \frac{\Gamma(n)L_n^{(\alpha)}(x)L_{n-1}^{(\alpha)}(x)}
        {\Gamma(n+\alpha)}
        \left(2+\frac{1}{n-1}\right)\\
      &\qquad=
        \sum_{n=N+1}^\infty
        \frac{\Gamma(n)x\left(L_n^{(\alpha)}(x)\right)^2}
        {\Gamma(n+\alpha+1)}
        -\frac{\Gamma(N+2)}{\Gamma(N+\alpha+1)}
        \frac{L_N^{(\alpha)}(x)L_{N+1}^{(\alpha)}(x)}{N}\;.
    \end{aligned}
  \end{align}
  By Fej\'er's formula, see Theorem~\ref{thm:fejer}, and due to the
  asymptotics~(\ref{eq:Gamma_asymp}), we can find $C\in\R$ such that,
  for all $n\in\N$ with $n\geq 2$ and for all $x\in[\eps, K]$,
  \begin{displaymath}
    \left|\frac{\Gamma(n)\left(L_n^{(\alpha)}(x)\right)^2}
      {\Gamma(n+\alpha+1)}\right|
    \leq Cn^{-\frac{3}{2}}
    \quad\text{as well as}\quad
    \left|\frac{\Gamma(n-1)L_n^{(\alpha)}(x)L_{n-1}^{(\alpha)}(x)}
      {\Gamma(n+\alpha)}\right|
    \leq Cn^{-\frac{3}{2}}
  \end{displaymath}
  and, for all $N\in\N$,
  \begin{displaymath}
    \left|\frac{\Gamma(N+2)}{\Gamma(N+\alpha+1)}
      \frac{L_N^{(\alpha)}(x)L_{N+1}^{(\alpha)}(x)}{N}\right|
    \leq C N^{-\frac{1}{2}}\;.
  \end{displaymath}
  From these estimates, by combining (\ref{eq:AL_help3}) as well as
  (\ref{eq:AL_help4}), and since $x\in[\eps,K]$, it follows that
  \begin{align*}
    &\left|\sqrt{N}\sum_{n=N+1}^\infty
      \frac{\Gamma(n)}{\Gamma(n+\alpha+1)}L_n^{(\alpha)}(x)
      \frac{\db L_n^{(\alpha)}(x)}{\db x}\right|\\
    &\qquad\leq\frac{\sqrt{N}C}{\eps}
    \left((\alpha+1)\sum_{n=N+1}^\infty n^{-\frac{3}{2}}
      +\sum_{n=N+1}^\infty n^{-\frac{3}{2}}
      +K\sum_{n=N+1}^\infty n^{-\frac{3}{2}}+N^{-\frac{1}{2}}\right)\\
    &\qquad\leq\frac{3(\alpha+3+K)C}{\eps}\;,
  \end{align*}
  as required.
\end{proof}
As in the proof of Lemma~\ref{lem:H_unifbound2}, we
note that simply applying the known asymptotic formula for the
polynomials does not give us sufficient control to obtain the desired
uniform bound in Lemma~\ref{lem:AL_unifbound2}. Instead, we first need
to perform a telescoping-like rearrangement. This approach is
motivated by a true telescoping series which appears when establishing
the corresponding uniform bound needed in the
proof of~\cite[Theorem~1.5]{semicircle}.
We stress that exactly the same idea is employed and
works in our subsequent analysis for the Jacobi polynomials, which
might not be immediately evident due to the lengthiness of the
expressions involved.

We are finally in a position to derive the asymptotic error in the
eigenfunction expansion for the Green's function corresponding to the
associated Laguerre polynomials.
\begin{proof}[Proof of Theorem~\ref{thm:laguerre}]
  We start by considering the convergence on the diagonal for which we
  study the functions $S_N\colon (0,\infty)\to\R$ given by, for $N\in\N_0$,
  \begin{displaymath}
    S_N(x)=\sqrt{N}\sum_{n=N+1}^\infty
    \frac{\Gamma(n)\left(L_n^{(\alpha)}(x)\right)^2}
    {\Gamma(n+\alpha+1)}
  \end{displaymath}
  in the limit $N\to\infty$. As a consequence of
  Lemma~\ref{lem:AL_unifbound2},
  we have, for all $x\in(0,\infty)$,
  \begin{displaymath}
    S_N'(x)= \sqrt{N}\sum_{n=N+1}^\infty
    \frac{2\,\Gamma(n)}{\Gamma(n+\alpha+1)}L_n^{(\alpha)}(x)
    \frac{\db L_n^{(\alpha)}(x)}{\db x}\;.
  \end{displaymath}
  Therefore, Lemma~\ref{lem:AL_unifbound1},
  Lemma~\ref{lem:AL_unifbound2} and the Arzel\`{a}--Ascoli theorem
  imply that the functions $S_N$ converge pointwise
  as $N\to\infty$ to the claimed
  limit function on $(0,\infty)$
  provided we establish the convergence of moments.
  We consider the moments defined by, for $k\in\N_0$,
  \begin{displaymath}
    p_k=\lim_{N\to\infty}\int_0^\infty
    x^kx^{2\alpha+1}\e^{-2x}S_N(x)\dd x\;,
  \end{displaymath}
  where the integrals are well-defined since $2\alpha+1>-1$. As
  positivity of the terms in the series allows us to interchange
  integration and summation, the moments defined above are related to
  the limit moments
  considered in Proposition~\ref{propn:moments} by
  \begin{displaymath}
    p_k=m_{k+1}\;.
  \end{displaymath}
  For all $l\in\N$, we observe that the function
  \begin{displaymath}
    x\mapsto \left(\left(x^lx^{\alpha+1}\e^{-x}\right)'
      x^{\alpha+1}\e^{-x}\right)'x^{-\alpha}\e^{x}
  \end{displaymath}
  is bounded on $(0,\infty)$ since $l+\alpha>0$. Using integration by
  parts as well as~(\ref{eq:AL_bc1}) and (\ref{eq:AL_bc2}), we can
  then argue that, as $n\to\infty$,
  \begin{align*}
    &\int_0^\infty\left(x^l x^{\alpha+1}\e^{-x}\right)'
    x^{\alpha+1}\e^{-x}L_n^{(\alpha)}(x)
    \frac{\db L_n^{(\alpha)}(x)}{\db x}\dd x\\
    &\qquad=-\frac{1}{2}\int_0^\infty
    \left(\left(x^lx^{\alpha+1}\e^{-x}\right)'x^{\alpha+1}\e^{-x}\right)'
      \left(L_n^{(\alpha)}(x)\right)^2\dd x
      =O\left(M_n\right)\;.
  \end{align*}
  While we do not consider the limit moment $m_0$, we still have to
  check that the condition~(\ref{eq:ac_COP}) in
  Proposition~\ref{propn:moments} is satisfied for $l=0$ since, as we
  see below, this relates the limit moments $m_1$ and $m_2$.
  Applying Lemma~\ref{lem:AL_ind} prior to integrating by parts, we
  obtain, for $n\in\N_0$,
  \begin{align*}
    &\int_0^\infty\left(x^{\alpha+1}\e^{-x}\right)'
    x^{\alpha+1}\e^{-x}L_n^{(\alpha)}(x)
    \frac{\db L_n^{(\alpha)}(x)}{\db x}\dd x\\
    &\qquad=(\alpha+1)\int_0^\infty x^{2\alpha+1}\e^{-2x}L_n^{(\alpha)}(x)
      \frac{\db L_n^{(\alpha)}(x)}{\db x}\dd x
      -\int_0^\infty x^{2\alpha+2}\e^{-2x}L_n^{(\alpha)}(x)
      \frac{\db L_n^{(\alpha)}(x)}{\db x}\dd x\\
    &\qquad
      =\frac{1}{2}\int_0^\infty
      \left(x^{2\alpha+2}\e^{-2x}\right)'
      \left(L_n^{(\alpha)}(x)\right)^2\dd x\;.
  \end{align*}
  Due to the function
  $x\mapsto \left(x^{2\alpha+2}\e^{-2x}\right)'x^{-\alpha}\e^{x}$ being
  bounded on $(0,\infty)$, it follows that, as $n\to\infty$,
  \begin{displaymath}
    \int_0^\infty\left(x^{\alpha+1}\e^{-x}\right)'
    x^{\alpha+1}\e^{-x}L_n^{(\alpha)}(x)
    \frac{\db L_n^{(\alpha)}(x)}{\db x}\dd x
    =O\left(M_n\right)\;,
  \end{displaymath}
  as needed. Together with (\ref{eq:AL_bc1}) and (\ref{eq:AL_bc2}),
  this means that Proposition~\ref{propn:moments} yields, for all
  $k\in\N_0$,
  \begin{displaymath}
    m_{k+2}=
    \left(\alpha+k+\frac{3}{2}\right)m_{k+1}
    \quad\text{and}\quad
    p_{k+1}=
    \left(\alpha+k+\frac{3}{2}\right)p_{k}\;.
  \end{displaymath}
  Even though we could show that Proposition~\ref{propn:conv_mom} gives rise
  to the same recurrence relation for suitably weighted moments of the
  claimed limit function for the full parameter range $\alpha>-1$, it
  is easier to
  establish this directly. We have, for all $k\in\N_0$,
  \begin{displaymath}
    \widetilde{p}_{k+1}=
    \frac{1}{\pi}\int_0^\infty x^{k+1}x^{\alpha+\frac{1}{2}}\e^{-x}\dd x
    =\frac{\alpha+k+\frac{3}{2}}{\pi}\int_0^\infty
    x^{k}x^{\alpha+\frac{1}{2}}\e^{-x}\dd x
    =\left(\alpha+k+\frac{3}{2}\right)\widetilde{p}_{k}\;.
  \end{displaymath}
  Thus, the desired pointwise convergence on the diagonal follows once
  we know that $p_0=\widetilde{p}_0$. From Lemma~\ref{lem:AL_initial}
  and the asymptotics~(\ref{eq:Gamma_asymp}) for the Gamma function,
  we see that, as $n\to\infty$,
  \begin{displaymath}
    \frac{\Gamma(n)}{\Gamma(n+\alpha+1)}
    \int_0^\infty x^{2\alpha+1}\e^{-2x}
    \left(L_{n}^{(\alpha)}(x)\right)^2\dd x
    =\frac{\Gamma(n)\Gamma\left(n+\frac{1}{2}\right)
      \Gamma\left(\alpha+\frac{3}{2}\right)}
    {2\pi\left(\Gamma(n+1)\right)^2}
    \sim\frac{\Gamma\left(\alpha+\frac{3}{2}\right)}
    {2\pi n^{\frac{3}{2}}}\;.
  \end{displaymath}
  It follows that
  \begin{displaymath}
    p_0=\lim_{N\to\infty}
    \sqrt{N}\sum_{n=N+1}^\infty
    \frac{\Gamma\left(\alpha+\frac{3}{2}\right)}
    {2\pi n^{\frac{3}{2}}}
    =\frac{\Gamma\left(\alpha+\frac{3}{2}\right)}{\pi}\;,
  \end{displaymath}
  as needed because of
  \begin{displaymath}
    \widetilde{p}_0=\frac{1}{\pi}\int_0^\infty
    x^{\alpha+\frac{1}{2}}\e^{-x}\dd x
    =\frac{\Gamma\left(\alpha+\frac{3}{2}\right)}{\pi}\;.
  \end{displaymath}
  We are left with proving the convergence away from the diagonal.
  Fix $x,y\in(0,\infty)$ with $x\not= y$.
  By Fej\'er's formula and due to the
  asymptotics~(\ref{eq:Gamma_asymp}), there
  exists a positive constant $C\in\R$ such that, for all $n\in\N$,
  \begin{equation}\label{eq:AL_diffbound}
    \left|\frac{\Gamma(n+2)}{\Gamma(n+\alpha+1)}
      \frac{L_{n+1}^{(\alpha)}(x)L_n^{(\alpha)}(y)-
        L_n^{(\alpha)}(x)L_{n+1}^{(\alpha)}(y)}{n}
    \right|
    \leq Cn^{-\frac{1}{2}}\;.
  \end{equation}
  Thus, applying
  Proposition~\ref{propn:CDT} to the associated Laguerre polynomials
  and using
  \begin{displaymath}
    \frac{1}{\lambda_{n}}-\frac{1}{\lambda_{n+1}}
    =\frac{1}{n}-\frac{1}{n+1}
    =\frac{1}{n(n+1)}\;,
  \end{displaymath}
  we obtain, for $N\in\N$,
  \begin{displaymath}
    (x-y)\sum_{n=N+1}^\infty
    \frac{\Gamma(n) L_n^{(\alpha)}(x)L_n^{(\alpha)}(y)}
    {\Gamma(n+\alpha+1)}
    =\frac{\Gamma(N+2)}{\Gamma(N+\alpha+1)}
    \frac{D_{N+1}(x,y)}{N}-\sum_{n=N}^\infty
    \frac{\Gamma(n)D_{n+1}(x,y)}{\Gamma(n+\alpha+1)}\;.
  \end{displaymath}
  Since~(\ref{eq:AL_diffbound}) gives
  \begin{displaymath}
    \left|\frac{\Gamma(n)D_{n+1}(x,y)}{\Gamma(n+\alpha+1)}\right|
    \leq Cn^{-\frac{3}{2}}\;,
  \end{displaymath}
  we deduce
  \begin{displaymath}
    \left|(x-y)\sum_{n=N+1}^\infty
      \frac{\Gamma(n) L_n^{(\alpha)}(x)L_n^{(\alpha)}(y)}
      {\Gamma(n+\alpha+1)}\right|
    \leq C\left(N^{-\frac{1}{2}}+\sum_{n=N}^\infty n^{-\frac{3}{2}}
    \right)
    \leq\frac{4C}{\sqrt{N}}\;.
  \end{displaymath}
  This shows that with $x\not= y$, for all $\gamma<\frac{1}{2}$,
  \begin{equation}\label{eq:AL_offdiag}
    \lim_{N\to\infty}N^\gamma
    \sum_{n=N+1}^\infty
    \frac{\Gamma(n) L_n^{(\alpha)}(x)L_n^{(\alpha)}(y)}
    {\Gamma(n+\alpha+1)}=0\;,
  \end{equation}
  as claimed.
\end{proof}
The reason why we obtain a weaker off-diagonal convergence result for the
associated Laguerre polynomials compared to the Hermite polynomials is
that Fej\'er's formula only allows us to obtain the
bound~(\ref{eq:AL_diffbound}) whereas for the Hermite polynomials we
have~(\ref{eq:H_diffbound}). It would be of interest to investigate
if~(\ref{eq:AL_offdiag}) remains true for all $\gamma<1$ provided
$x\not=y$.

\subsection{Jacobi polynomials}
\label{sec:jacobi}

The family $\{P_n^{(\alpha,\beta)}:n\in\N_0\}$ of Jacobi polynomials
on $[-1,1]$ is given
in terms of two parameters $\alpha$ and $\beta$.
While the Jacobi
polynomials can be defined for all $\alpha,\beta\in\R$ via
the rising Pochhammer symbol by, for $n\in\N_0$ and $x\in[-1,1]$,
\begin{displaymath}
  P_n^{(\alpha,\beta)}(x)
  =\frac{1}{n!}\sum_{k=0}^n\binom{n}{k}
  \left(n+\alpha+\beta+1\right)_k\left(\alpha+k+1\right)_{n-k}
  \left(\frac{x-1}{2}\right)^k\;,
\end{displaymath}
see~\cite[Section~4.22]{szego}, we only obtain
classical orthogonal polynomials for $\alpha,\beta>-1$. In that
regime,
the Jacobi polynomials are the classical orthogonal
polynomials on $I=[-1,1]$ with respect to the weight function
$W\colon (-1,1)\to\R$ given by
\begin{displaymath}
  W(x)=(1-x)^\alpha(1+x)^\beta
\end{displaymath}
subject to the standardisation
\begin{displaymath}
  P_n^{(\alpha,\beta)}(1)
  =\frac{\Gamma(n+\alpha+1)}{n!\,\Gamma(\alpha+1)}\;.
\end{displaymath}
The class of Jacobi
polynomials is a rich class of orthogonal polynomials which gives
rise, for instance,
to the Legendre polynomials $\{P_n:n\in\N_0\}$,
to the Chebyshev polynomials of the first kind $\{T_n:n\in\N_0\}$ and
to the
Chebyshev polynomials of the second kind $\{U_n:n\in\N_0\}$ through,
for $n\in\N_0$,
\begin{equation}\label{eq:J_special}
  P_n=P_n^{(0,0)}\;,\quad
  T_n=\frac{n!\sqrt{\pi}}
  {\Gamma\left(n+\frac{1}{2}\right)}
  P_n^{(-\frac{1}{2},-\frac{1}{2})}
  \quad\text{and}\quad
  U_n=\frac{(n+1)!\sqrt{\pi}}
  {2\,\Gamma\left(n+\frac{3}{2}\right)}
  P_n^{(\frac{1}{2},\frac{1}{2})}\;.
\end{equation}
If $\alpha,\beta>-1$, the weight function $W$ is integrable on
$(-1,1)$ and, as derived for~\cite[4.3.3]{szego}, we obtain, for
$n\in\N$,
\begin{equation}\label{eq:J_defnM}
  M_n=\frac{2^{\alpha+\beta+1}\Gamma(n+\alpha+1)\Gamma(n+\beta+1)}
  {n!(2n+\alpha+\beta+1)\Gamma(n+\alpha+\beta+1)}\;.
\end{equation}
Note that if $\alpha,\beta\in\R$, the above expression remains
well-defined for sufficiently large $n\in\N$, and we can
take~(\ref{eq:J_defnM}) as the definition for $M_n$
as long as $n+1>-\min(\alpha,\beta,\alpha+\beta)$
whenever we are outside
the regime $\alpha,\beta>-1$. Moreover, for $\alpha,\beta\in\R$ fixed,
the Jacobi polynomials satisfy, for $n\in\N_0$,
\begin{displaymath}
  K_n=\frac{1}{2^n}\binom{2n+\alpha+\beta}{n}\;,
\end{displaymath}
see~\cite[4.21.6]{szego}, and they solve the
differential equation, for $x\in[-1,1]$,
\begin{displaymath}
  \left(1-x^2\right)\,\frac{\db^2P_n^{(\alpha,\beta)}(x)}{\db x^2}+
  \left(\beta-\alpha-(\alpha+\beta+2)x\right)\,
  \frac{\db P_n^{(\alpha,\beta)}(x)}{\db x}+
  n(n+\alpha+\beta+1) P_n^{(\alpha,\beta)}(x)=0\;,
\end{displaymath}
see~\cite[Theorem 4.2.2]{szego}. This implies that
\begin{displaymath}
  Q(x)=1-x^2\;,\quad
  L(x)=\beta-\alpha-(\alpha+\beta+2)x
  \quad\text{and}\quad
  \lambda_n=n(n+\alpha+\beta+1)\;.
\end{displaymath}
Similar to our study of the asymptotic error in the eigenfunction expansion
for the Green's function associated with the Hermite polynomials and
with the associated Laguerre polynomials, we could use
Proposition~\ref{propn:moments} and Proposition~\ref{propn:conv_mom}
to show that for Jacobi polynomials, certainly if
$\alpha,\beta\geq 0$, the limit moments on the
diagonal satisfy the same recurrence relation as the weighted moments
of the claimed limit function. However, this approach would require us
to obtain asymptotic control, as $n\to\infty$, over
\begin{displaymath}
  \int_{-1}^1(1-x)^{2\alpha}(1+x)^{2\beta}
  \left(P_n^{(\alpha,\beta)}(x)\right)^2\dd x
  \quad\text{and}\quad
  \int_{-1}^1x(1-x)^{2\alpha}(1+x)^{2\beta}
  \left(P_n^{(\alpha,\beta)}(x)\right)^2\dd x\;,
\end{displaymath}
which does not appear to be straightforward. Instead, we exploit the asymptotic
formula which is known for the Jacobi polynomials more than we did use
the asymptotic formulae for the
Hermite polynomials and for the associated Laguerre polynomials. The
benefits of this approach are that it links to the analysis performed
in Section~\ref{sec:regular} and that it works for all
$\alpha,\beta\in\R$. In particular, no extra care is needed when
$\alpha,\beta\in(-1,-\frac{1}{2}]$ to ensure integrability of the
integrals we consider.
The Jacobi polynomials admit the following asymptotic
formula stated in~\cite[Theorem~8.21.8]{szego}, which is due to
Darboux~\cite{darboux}.
\begin{thm}[Darboux formula]\label{thm:darboux}
  Let $\alpha,\beta\in\R$ be fixed.
  For $\theta\in(0,\pi)$, set
  \begin{equation*}
    k(\theta)=\pi^{-\frac{1}{2}}
    \left(\sin\left(\frac{\theta}{2}\right)\right)^{-\alpha-\frac{1}{2}}
    \left(\cos\left(\frac{\theta}{2}\right)\right)^{-\beta-\frac{1}{2}}\;.
  \end{equation*}
  We have, as $n\to\infty$,
  \begin{equation*}
    P_n^{(\alpha,\beta)}\left(\cos\left(\theta\right)\right)=
    n^{-\frac{1}{2}}k(\theta)
    \cos\left(\left(n+\frac{\alpha+\beta+1}{2}\right)\theta
      -\left(\alpha+\frac{1}{2}\right)\frac{\pi}{2}\right)
    +O\left(n^{-\frac{3}{2}}\right)\;,
  \end{equation*}
  where the bound on the error term is uniform
  in $\theta\in[\eps,\pi-\eps]$ for $\eps>0$.
\end{thm}
To extend our considerations to $\alpha,\beta\in\R$, we further need
to observe that the Jacobi polynomials satisfy the three-term
recurrence relation~\cite[4.5.1]{szego} for all
$\alpha,\beta\in\R$. This implies that the proof of
Proposition~\ref{propn:CDT} can be adapted to give us sufficient
control to deduce the off-diagonal convergence
for all families
of Jacobi polynomials.

Besides, as derived for~\cite[4.21.7]{szego}, we have the identity that,
for all $n\in\N$ and for all $x\in[-1,1]$,
\begin{equation}\label{eq:J_diff}
  \frac{\db}{\db x}P_n^{(\alpha,\beta)}(x)
  =\frac{1}{2}\left(n+\alpha+\beta+1\right)
  P_{n-1}^{(\alpha+1,\beta+1)}(x)\;.
\end{equation}
As previously, we still require two local uniform bounds to establish the
on-diagonal convergence in Theorem~\ref{thm:jacobi}, and one of the two
bounds is
obtained through a telescoping-like rearrangement relying on the next
lemma.
\begin{lemma}\label{lem:J_diff}
  For $\alpha,\beta\in\R$ fixed, we have, for
  all $n\in\N$ with $n\geq 2$ and for all $x\in[-1,1]$,
  \begin{align*}
    &(2n+\alpha+\beta)(2n+\alpha+\beta+1)
    (2n+\alpha+\beta+2)P_n^{(\alpha,\beta)}(x)\\
    &\qquad=(2n+\alpha+\beta)(n+\alpha+\beta+1)(n+\alpha+\beta+2)
      P_n^{(\alpha+1,\beta+1)}(x)\\
    &\qquad\qquad+(2n+\alpha+\beta+1)(\alpha-\beta)(n+\alpha+\beta+1)
      P_{n-1}^{(\alpha+1,\beta+1)}(x)\\
    &\qquad\qquad-(2n+\alpha+\beta+2)(n+\alpha)(n+\beta)
      P_{n-2}^{(\alpha+1,\beta+1)}(x)\;.
  \end{align*}
\end{lemma}
\begin{proof}
  According to~\cite[\S 138. (14), (15)]{rainville}, we have, for all
  $n\in\N$ and for all $x\in[-1,1]$,
  \begin{equation}\label{eq:J_shift1}
    (2n+\alpha+\beta+1)P_n^{(\alpha,\beta)}(x)
    =(n+\alpha+\beta+1)P_n^{(\alpha,\beta+1)}(x)
    +(n+\alpha)P_{n-1}^{(\alpha,\beta+1)}(x)
  \end{equation}
  as well as
  \begin{equation}\label{eq:J_shift2}
    (2n+\alpha+\beta+1)P_n^{(\alpha,\beta)}(x)
    =(n+\alpha+\beta+1)P_n^{(\alpha+1,\beta)}(x)
    -(n+\beta)P_{n-1}^{(\alpha+1,\beta)}(x)\;.
  \end{equation}
  From~(\ref{eq:J_shift2}), it follows that, for $n\in\N$,
  \begin{displaymath}
    (2n+\alpha+\beta+2)P_n^{(\alpha,\beta+1)}(x)
    =(n+\alpha+\beta+2)P_n^{(\alpha+1,\beta+1)}(x)
    -(n+\beta+1)P_{n-1}^{(\alpha+1,\beta+1)}(x)
  \end{displaymath}
  and, for $n\in\N$ with $n\geq 2$,
  \begin{displaymath}
    (2n+\alpha+\beta)P_{n-1}^{(\alpha,\beta+1)}(x)
    =(n+\alpha+\beta+1)P_{n-1}^{(\alpha+1,\beta+1)}(x)
    -(n+\beta)P_{n-2}^{(\alpha+1,\beta+1)}(x)\;.
  \end{displaymath}
  Substituting these two identities into~(\ref{eq:J_shift1}) yields
  the claimed result.
\end{proof}
We observe that, due to the asymptotics~(\ref{eq:Gamma_asymp}) for the
Gamma function, as $n\to\infty$,
\begin{displaymath}
  M_n \sim \frac{2^{\alpha+\beta}}{n}\;.
\end{displaymath}
In particular, the two local uniform bounds derived below are
sufficient for our purposes.
\begin{lemma}\label{lem:J_unifbound1}
  Fix $\eps\in\R$ with $\eps>0$ and fix $\alpha,\beta\in\R$. The family
  \begin{align*}
    \left\{N\sum_{n=N+1}^\infty
      \frac{P_n^{(\alpha,\beta)}(x)P_n^{(\alpha,\beta)}(y)}{n}:
      N\in\N\text{ and }x,y\in[-1+\eps,1-\eps]
    \right\}
  \end{align*}
  is uniformly bounded.
\end{lemma}
\begin{proof}
  As a consequence of the Darboux formula, see
  Theorem~\ref{thm:darboux}, there exists a positive constant $C\in\R$ such
  that, for all $n\in\N$ and for all $x,y\in[-1+\eps,1-\eps]$,
  \begin{displaymath}
    \left|P_n^{(\alpha,\beta)}(x)P_n^{(\alpha,\beta)}(y)\right|
    \leq Cn^{-1}\;.
  \end{displaymath}
  We conclude that, for all $N\in\N$,
  \begin{displaymath}
    \left|N\sum_{n=N+1}^\infty
      \frac{P_n^{(\alpha,\beta)}(x)P_n^{(\alpha,\beta)}(y)}{n}\right|
    \leq N\sum_{n=N+1}^\infty \frac{C}{n^{2}}\leq C\;,
  \end{displaymath}
  as required.
\end{proof}
\begin{lemma}\label{lem:J_unifbound2}
  Fix $\eps\in\R$ with $\eps>0$ and fix $\alpha,\beta\in\R$. The family
  \begin{align*}
    \left\{N\sum_{n=N+1}^\infty
      P_n^{(\alpha,\beta)}(x)P_{n-1}^{(\alpha+1,\beta+1)}(x):
      N\in\N\text{ and }x\in[-1+\eps,1-\eps]
    \right\}
  \end{align*}
  is uniformly bounded.
\end{lemma}
\begin{proof}
  Since the Darboux formula by itself does not provide sufficient control to
  deduce the desired uniform boundedness, we first employ a similar
  telescoping-like trick as before and then apply the Darboux
  formula. Using Lemma~\ref{lem:J_diff}, we write, for
  $n\in\N$ with $2n+\alpha+\beta>0$ and for $x\in[-1,1]$,
  \begin{align}\label{eq:J_diff2}
    \begin{aligned}
      &P_n^{(\alpha,\beta)}(x)P_{n-1}^{(\alpha+1,\beta+1)}(x)\\
      &\qquad=\frac{(n+\alpha+\beta+1)(n+\alpha+\beta+2)}
        {(2n+\alpha+\beta+1)(2n+\alpha+\beta+2)}
        P_n^{(\alpha+1,\beta+1)}(x)P_{n-1}^{(\alpha+1,\beta+1)}(x)\\
      &\qquad\qquad+\frac{(\alpha-\beta)(n+\alpha+\beta+1)}
        {(2n+\alpha+\beta)(2n+\alpha+\beta+2)}
        \left(P_{n-1}^{(\alpha+1,\beta+1)}(x)\right)^2\\
      &\qquad\qquad-\frac{(n+\alpha)(n+\beta)}
        {(2n+\alpha+\beta)(2n+\alpha+\beta+1)}
        P_{n-2}^{(\alpha+1,\beta+1)}(x)P_{n-1}^{(\alpha+1,\beta+1)}(x)\;.
    \end{aligned}
  \end{align}
  According to the Darboux formula, there exists a positive
  constant $C\in\R$ such
  that, for all $N\in\N$ with $2N+\alpha+\beta>0$ and for all
  $x\in[-1+\eps,1-\eps]$,
  \begin{displaymath}
    \left|N\sum_{n=N+1}^\infty
      \frac{(\alpha-\beta)(n+\alpha+\beta+1)}
        {(2n+\alpha+\beta)(2n+\alpha+\beta+2)}
        \left(P_{n-1}^{(\alpha+1,\beta+1)}(x)\right)^2\right|
    \leq N\sum_{n=N+1}^\infty \frac{C}{n^{2}}\leq C\;.
  \end{displaymath}
  For the remaining two series arising from~(\ref{eq:J_diff2}), we
  observe that, as $n\to\infty$,
  \begin{displaymath}
    \frac{(n+\alpha+\beta+1)(n+\alpha+\beta+2)}
    {(2n+\alpha+\beta+1)(2n+\alpha+\beta+2)}
    -\frac{(n+\alpha+1)(n+\beta+1)}
    {(2n+\alpha+\beta+2)(2n+\alpha+\beta+3)}
    \sim \frac{\alpha+\beta+2}{4n}
  \end{displaymath}
  and that, due to the Darboux formula, the constant
  $C\in\R$ can be chosen to satisfy, for all $N\in\N$ with
  $2N+\alpha+\beta>0$ and for all $x\in[-1+\eps,1-\eps]$,
  \begin{displaymath}
    \left|\frac{(N+\alpha+1)(N+\beta+1)}
      {(2N+\alpha+\beta+2)(2N+\alpha+\beta+3)}
      P_{N-1}^{(\alpha+1,\beta+1)}(x)P_{N}^{(\alpha+1,\beta+1)}(x)
    \right|\leq C N^{-1}
  \end{displaymath}
  as well as
  \begin{displaymath}
    \left|N\sum_{n=N+1}^\infty\frac{(\alpha+\beta+2)
        P_n^{(\alpha+1,\beta+1)}(x)P_{n-1}^{(\alpha+1,\beta+1)}(x)}{4n}\right|
    \leq C\;.
  \end{displaymath}
  Considering terms of the original series which correspond to $n\in\N$
  with $2n+\alpha+\beta\leq0$ separately,
  we conclude there exists a positive constant $D\in\R$ such that, for all
  $N\in\N$ and $x\in[-1+\eps,1-\eps]$,
  \begin{displaymath}
    \left|N\sum_{n=N+1}^\infty
      P_n^{(\alpha,\beta)}(x)P_{n-1}^{(\alpha+1,\beta+1)}(x)\right|
    \leq 3C+D\;.
  \end{displaymath}
  This establishes the claimed uniform boundedness.
\end{proof}
We are now in a position to prove the proposition stated below. While
it has some overlap with Theorem~\ref{thm:regularSL}, for most
$\alpha,\beta\in\R$ the result is not a straightforward
application of Theorem~\ref{thm:regularSL}. This is easily seen by
noting that
\begin{displaymath}
  \pi\left(n+\frac{\alpha+\beta+1}{2}\right)
\end{displaymath}
need not satisfy the asymptotics given in Proposition~\ref{propn:evals} as
$n\to\infty$. In particular, we move outside the setting of separated
homogeneous boundary conditions. The reason for having delayed the
proof of the next proposition until now is that it seems more
convenient to use Lemma~\ref{lem:J_unifbound2} for the second local uniform
bound as opposed to carefully integrating by parts as in the proof of
Proposition~\ref{propn:baseDN} and as for
establishing~\cite[Lemma~4.2]{foster_habermann}.
\begin{propn}\label{propn:generalbc}
  Fix $\alpha,\beta\in\R$. We have, for $\theta\in(0,\pi)$,
  \begin{displaymath}
    \lim_{N\to\infty}N\sum_{n=N+1}^\infty
    \frac{2\left(\cos\left(\left(n+\dfrac{\alpha+\beta+1}{2}\right)\theta
          -\left(\alpha+\dfrac{1}{2}\right)\dfrac{\pi}{2}\right)\right)^2}
    {n^2}
    =1\;.
  \end{displaymath}
\end{propn}
\begin{proof}
  For $N\in\N_0$, let $S_N\colon(-1,1)\to\R$ be defined by
  \begin{displaymath}
    S_N(x)=N\sum_{n=N+1}^\infty
    \frac{\left(P_n^{(\alpha,\beta)}(x)\right)^2}{n}\;.
  \end{displaymath}
  Due to the identity~(\ref{eq:J_diff}) for the derivative of Jacobi
  polynomials, it is a consequence of 
  Lemma~\ref{lem:J_unifbound2} that, for all $x\in(-1,1)$,
  \begin{displaymath}
    S_N'(x)=N\sum_{n=N+1}^\infty
    \frac{n+\alpha+\beta+1}{n}P_n^{(\alpha,\beta)}(x)
      P_{n-1}^{(\alpha+1,\beta+1)}(x)\;.
  \end{displaymath}
  The Arzel\`{a}--Ascoli theorem together with
  Lemma~\ref{lem:J_unifbound1} as well as Lemma~\ref{lem:J_unifbound2}
  then imply that the functions $S_N$ converge
  locally uniformly along subsequences on the interval $(-1,1)$ as
  $N\to\infty$. From the
  Darboux formula, it follows that with $u_n\colon[0,\pi]\to\R$ for
  $n\in\N_0$ given by
  \begin{displaymath}
    u_n(\theta)=\cos\left(\left(n+\dfrac{\alpha+\beta+1}{2}\right)\theta
      -\left(\alpha+\dfrac{1}{2}\right)\dfrac{\pi}{2}\right)\;,
  \end{displaymath}
  the functions $R_N\colon(0,\pi)\to\R$ defined by
  \begin{displaymath}
    R_N(\theta)=N\sum_{n=N+1}^\infty
    \frac{2\left(u_n(\theta)\right)^2}{n^2}
  \end{displaymath}
  converge locally uniformly  along subsequences on $(0,\pi)$ as $N\to\infty$.
  Hence, we can apply a moment argument to show that the functions
  $R_N$ converge pointwise as $N\to\infty$ and to identify the
  continuous limit function on $(0,\pi)$.
  Set, for $n\in\N_0$,
  \begin{displaymath}
    \nu_n=\left(n+\dfrac{\alpha+\beta+1}{2}\right)^2\;.
  \end{displaymath}
  Proceeding as in the proof of
  Proposition~\ref{propn:baseDN}, we obtain, for all $n,k\in\N_0$,
  \begin{displaymath}
    \nu_n\int_0^\pi\theta^k\left(u_n(\theta)\right)^2\dd\theta
    =-\left.\theta^ku_n(\theta)u_n'(\theta)\right|_0^\pi
    +\int_0^\pi\theta^k\left(u_n'(\theta)\right)^2\dd\theta
    +\int_0^\pi k\theta^{k-1}u_n(\theta)u_n'(\theta)\dd\theta
  \end{displaymath}
  as well as
  \begin{displaymath}
    \int_0^\pi\theta^k\left(u_n'(\theta)\right)^2\dd\theta
    +\nu_n\int_0^\pi\theta^k\left(u_n(\theta)\right)^2\dd\theta
    =\frac{\pi^{k+1}\left(\left(u_n'(\pi)\right)^2
        +\nu_n\left(u_n(\pi)\right)^2\right)}{k+1}
    =\frac{\pi^{k+1}\nu_n}{k+1}\;.
  \end{displaymath}
  Since we have, for $k\in\N_0$ fixed and as $n\to\infty$,
  \begin{displaymath}
    \left.\theta^ku_n(\theta)u_n'(\theta)\right|_0^\pi
    =O(n)\quad\text{and}\quad
    \int_0^\pi k\theta^{k-1}u_n(\theta)u_n'(\theta)\dd\theta
    =O(1)\;,
  \end{displaymath}
  it follows that, for $k\in\N_0$ fixed and as $n\to\infty$,
  \begin{displaymath}
    2\nu_n\int_0^\pi\theta^k\left(u_n(\theta)\right)^2\dd\theta
    =\frac{\pi^{k+1}\nu_n}{k+1}+O(n)\;.
  \end{displaymath}
  Observing that $\nu_n=O\left(n^2\right)$ as $n\to\infty$,
  we deduce
  \begin{displaymath}
    \int_0^\pi\frac{2\theta^k\left(u_n(\theta)\right)^2}{n^2}\dd\theta
    =\frac{\pi^{k+1}}{(k+1)n^2}+O\left(n^{-3}\right)\;,
  \end{displaymath}
  which by Fubini's theorem implies that, for all $k\in\N_0$
  \begin{displaymath}
    \lim_{N\to\infty}\int_0^\pi\theta^kR_N(\theta)\dd\theta
    =\frac{\pi^{k+1}}{k+1}=\int_0^\pi \theta^k\dd\theta\;.
  \end{displaymath}
  As the limit moments agree with the moments of the claimed limit
  function, the result follows by continuity.
\end{proof}
We close by proving Theorem~\ref{thm:jacobi} on the asymptotic error
in the eigenfunction expansion for the Green's function associated
with the Jacobi polynomials and by showing
that~\cite[Theorem~1.5]{semicircle} is a special case thereof.
\begin{proof}[Proof of Theorem~\ref{thm:jacobi}]
  Recall that due to the asymptotics~(\ref{eq:Gamma_asymp}) for the
  Gamma function, we have, as $n\to\infty$,
  \begin{displaymath}
    M_n\sim\frac{2^{\alpha+\beta}}{n}\;.
  \end{displaymath}
  Moreover, we note that, for $\theta\in(0,\pi)$ and $x=\cos(\theta)$,
  \begin{displaymath}
    \left(\sin\left(\frac{\theta}{2}\right)\right)^{-2\alpha-1}
    \left(\cos\left(\frac{\theta}{2}\right)\right)^{-2\beta-1}
    =2^{\alpha+\beta+1}(1-x)^{-\alpha-\frac{1}{2}}(1+x)^{-\beta-\frac{1}{2}}\;.
  \end{displaymath}
  The Darboux formula and Proposition~\ref{propn:generalbc} then
  yield, for all $x\in(-1,1)$,
  \begin{displaymath}
    \lim_{N\to\infty}N\sum_{n=N+1}^\infty
    \frac{\left(P_n^{(\alpha,\beta)}(x)\right)^2}{M_n\lambda_n}
    =\frac{(1-x)^{-\alpha-\frac{1}{2}}(1+x)^{-\beta-\frac{1}{2}}}{\pi}\;,
  \end{displaymath}
  which establishes the claimed convergence on the diagonal. For the
  convergence away from the diagonal, we exploit a shifted
  Christoffel--Darboux type formula which ensures that all terms are
  well-defined. For the Jacobi polynomials, we have, by Stirling's
  formula~(\ref{eq:stirling}) and as $n\to\infty$,
  \begin{displaymath}
    K_n=\frac{\Gamma(2n+\alpha+\beta+1)}
    {2^n n!\, \Gamma(n+\alpha+\beta+1)}
    \sim\frac{2^{n+\alpha+\beta}}{\sqrt{\pi n}}\;.
  \end{displaymath}
  Using the Darboux formula, we deduce that there exists a constant
  $C\in\R$ such 
  that, for sufficiently large $n\in\N$ and for $x,y\in(-1,1)$,
  \begin{equation}\label{eq:J_CDTlim}
    \left|\frac{K_n}{K_{n+1}M_n}\frac{D_{n+1}(x,y)}{\lambda_n}\right|
    \leq C n^{-2}\;.
  \end{equation}
  Since the Jacobi polynomials satisfy the three-term
  recurrence relation~\cite[4.5.1]{szego} for all $\alpha,\beta\in\R$,
  we can adapt the proof of Proposition~\ref{propn:CDT}, by choosing
  the starting point of the sum large enough to ensure that all
  summands are well-defined, to show
  that, for sufficiently large $N\in\N$,
  \begin{align}\label{eq:J_CDT}
    \begin{aligned}
      &(x-y)\sum_{n=N+1}^\infty
      \frac{P_n^{(\alpha,\beta)}(x)P_n^{(\alpha,\beta)}(y)}{M_n\lambda_n}\\
      &\qquad=\sum_{n=N}^\infty\frac{K_{n}}{K_{n+1}M_{n}}D_{n+1}(x,y)
      \left(\frac{1}{\lambda_{n}}-\frac{1}{\lambda_{n+1}}\right)
      -\frac{K_N}{K_{N+1}M_N}\frac{D_{N+1}(x,y)}{\lambda_N}\;.
    \end{aligned}
  \end{align}
  We further have, as $n\to\infty$,
  \begin{displaymath}
    \frac{1}{\lambda_n}-\frac{1}{\lambda_{n+1}}
    =\frac{2n+\alpha+\beta+2}
    {n(n+1)(n+\alpha+\beta+1)(n+\alpha+\beta+2)}
    \sim\frac{2}{n^3}\;.
  \end{displaymath}
  Hence, by the Darboux formula, the constant $C\in\R$ can be chosen
  such that, for
  sufficiently large $n\in\N$ and for $x,y\in(-1,1)$,
  \begin{displaymath}
    \left|\frac{K_n}{K_{n+1}M_n}D_{n+1}(x,y)
      \left(\frac{1}{\lambda_n}-\frac{1}{\lambda_{n+1}}\right)\right|
    \leq Cn^{-3}\;,
  \end{displaymath}
  which together with~(\ref{eq:J_CDTlim}) and (\ref{eq:J_CDT}) implies
  that, for $x\not=y$ and for all $\gamma<2$,
  \begin{displaymath}
    \lim_{N\to\infty}N^\gamma \sum_{n=N+1}^\infty
      \frac{P_n^{(\alpha,\beta)}(x)P_n^{(\alpha,\beta)}(y)}{M_n\lambda_n}=0\;,
  \end{displaymath}
  as claimed.
\end{proof}
Note that Corollary~\ref{cor:jacobi} is an immediate consequence of
Theorem~\ref{thm:jacobi} because of~(\ref{eq:J_special}), which gives, for
$n\in\N$ and $x,y\in(-1,1)$,
\begin{displaymath}
  \frac{T_n(x)T_n(y)}{n^2}
  =\frac{\pi}{2}\left(\frac{n!(2n)(n-1)!}
  {\left(\Gamma\left(n+\frac{1}{2}\right)\right)^2}
  \frac{P_n^{(-\frac{1}{2},-\frac{1}{2})}(x)
  P_n^{(-\frac{1}{2},-\frac{1}{2})}(y)}{n^2}\right)
\end{displaymath}
as well as
\begin{displaymath}
  \frac{U_n(x)U_n(y)}{n(n+2)}
  =\frac{\pi}{2}\left(\frac{n!(2n+2)(n+1)!}
  {4\left(\Gamma\left(n+\frac{3}{2}\right)\right)^2}
  \frac{P_n^{(\frac{1}{2},\frac{1}{2})}(x)
  P_n^{(\frac{1}{2},\frac{1}{2})}(y)}{n(n+2)}\right)\;.
\end{displaymath}
From Theorem~\ref{thm:jacobi}, we can further
deduce~\cite[Theorem~1.5]{semicircle} which states that, for $x,y\in[-1,1]$, 
\begin{displaymath}
  \lim_{N\to\infty}N\sum_{n=N}^\infty\frac{2n+1}{2}
      \int_{-1}^xP_n(z)\dd z
      \int_{-1}^yP_n(z)\dd z=
  \begin{cases}
    \dfrac{\sqrt{1-x^2}}{\pi} & \mbox{if } x=y\\[0.7em]
    0 & \mbox{if } x\not=y
  \end{cases}\;.
\end{displaymath}
Due to~(\ref{eq:J_special}) and
(\ref{eq:J_diff}), we have, for $n\in\N$ and $x\in[-1,1]$,
\begin{displaymath}
  \frac{n-1}{2}\int_{-1}^xP_{n-1}(z)\dd z
  =P_{n}^{(-1,-1)}(x)\;,
\end{displaymath}
which yields, for $N\in\N$,
\begin{displaymath}
  N\sum_{n=N}^\infty\frac{2n+1}{2}
      \int_{-1}^xP_n(z)\dd z
      \int_{-1}^yP_n(z)\dd z=
  N\sum_{n=N+1}^\infty\frac{2(2n-1)}{(n-1)^2}
      P_{n}^{(-1,-1)}(x)P_{n}^{(-1,-1)}(y)\;.
\end{displaymath}
This together with Theorem~\ref{thm:jacobi} implies the previous result
since, for $n\in\N$ with $n\geq 2$,
\begin{displaymath}
  \frac{2n!(2n-1)(n-2)!}{(n-1)!(n-1)!}\frac{1}{n(n-1)}
  =\frac{2(2n-1)}{(n-1)^2}\;,
\end{displaymath}
and as $P_{n}^{(-1,-1)}(-1)=P_{n}^{(-1,-1)}(1)=0$ for all $n\in\N$.
%%%%%%%%%%%%%%%%%%%%%%%%%%%%%%%%%%%%%%%%%%%%%%%%%%%%%%%%%%%%%%%%%%%%%%%%%%%%%%% 
\bibliographystyle{plain}
\bibliography{references}

\end{document}